\documentclass[a4paper,11pt]{article}

\usepackage[T1]{fontenc}
\usepackage{lmodern}

\usepackage{dsfont}

\usepackage[latin1]{inputenc}
\usepackage{amsmath}
\usepackage{amsthm}
\usepackage{amssymb}
\usepackage{mathrsfs}
\usepackage{graphicx}
\usepackage[all]{xy}
\usepackage{hyperref}
\usepackage{xcolor}
\usepackage{tikz}

\usepackage{makeidx}

\usepackage{stmaryrd}
\usepackage{caption}

\usepackage{abstract} 

\newtheorem{thm}{Theorem}[section]
\newtheorem{cor}[thm]{Corollary}
\newtheorem{claim}[thm]{Claim}
\newtheorem{fact}[thm]{Fact}

\newtheorem{lemma}[thm]{Lemma}
\newtheorem{prop}[thm]{Proposition}

\theoremstyle{definition}
\newtheorem{definition}[thm]{Definition}

\newtheorem{remark}[thm]{Remark}

\def\rquotient#1#2{%
	\makeatletter
	\raise.3ex\hbox{$#1$}/\lower.3ex\hbox{$#2$}%
	\makeatother
}	

\makeatletter
\newcommand{\subjclass}[2][2010]{%
	\let\@oldtitle\@title%
	\gdef\@title{\@oldtitle\footnotetext{#1 \emph{Mathematics subject classification.} #2}}%
}
\newcommand{\keywords}[1]{%
	\let\@@oldtitle\@title%
	\gdef\@title{\@@oldtitle\footnotetext{\emph{Key words and phrases.} #1.}}%
}
\makeatother

\newcommand{\Address}{{
		\bigskip
		\small

\noindent\textsc{University of Montpellier\\ 
Institut Math\'ematiques Alexander Grothendieck\\
Place Eug\`ene Bataillon\\
34090 Montpellier (France)\\}\nopagebreak
\textit{E-mail address}: \texttt{anthony.genevois@umontpellier.fr}

\medskip
		
\noindent\textsc{Univ. Lyon \\
ENS de Lyon \\
UMPA UMR 5669 \\
46 all{\'e}e d'Italie \\
F-69364 Lyon cedex 07\\}\nopagebreak \textit{E-mail address}: \texttt{yassine.guerch@ens-lyon.fr}

\medskip

		\noindent\textsc{Institut de Math\'ematiques de Jussieu-Paris Rive Gauche, \\ Place Aur\'elie Nemours, \\ 75013 Paris, France.}\par\nopagebreak
		\noindent\textit{E-mail address}: \texttt{romain.tessera@imj-prg.fr}
\medskip
}}

\makeindex

\title{Folding median graphs}
\date{\today}
\author{Anthony Genevois, Yassine Guerch, and Romain Tessera}


\begin{document}

\maketitle

\begin{abstract}
Extending Stallings' foldings of trees, we show in this article that every \emph{parallel-preserving} map between median graphs factors as an isometric embedding through a sequence of elementary transformations which we call \emph{foldings} and \emph{swellings}. This new construction proposes a unified point of view on Beeker and Lazarovich's work on folding pocsets and on Ben-Zvi, Kropholler, and Lyman's work on folding nonpositively curved cube complexes. 
\end{abstract}

\tableofcontents

\section{Introduction}

\medskip \noindent
In his seminal work \cite{stallings1983topology}, Stallings introduced \emph{folds} in the study of free groups. The central idea is that, for every edge-preserving map $\psi : T_1 \to T_2$ between two trees $T_1,T_2$, there exists a(n infinite) sequence of folds $\eta : T_1 \to \cdots \to T_3$ and an isometric embedding $\iota : T_3 \to T_2$ such that $\psi = \iota \circ \eta$. Here, a fold refers to the identification of two intersecting edges. Since then, Stallings' folds have been applied to more general actions on trees~\cite{stallings1991foldings}; and they have been adapted to other groups such as random and small cancellation groups \cite{MR1445193, MR1869228}, some Coxeter groups \cite{MR1956839, MR2018959}, automatic groups~\cite{steinberg2001finite,kharlampovich2017stallings}, Thompson's group $F$ \cite{MR3710646}, right-angled Coxeter groups~\cite{dani2021subgroups}. Stallings' folds have also been applied for proving structural results on \emph{deformation spaces}, which can be thought of as analogues of the Teichm{\"u}ller space in the case of groups acting on trees (see for instance~\cite{skoradeformation,clay2005contractibility, guirardel2007deformation}). 

\medskip \noindent
In this article, we are interested in generalising Stallings' folds to median graphs (also known as one-skeletons of CAT(0) cube complexes\footnote{Median graphs and CAT(0) cube complexes essentially define the same object. Following \cite{MedianVs}, we use the terminology of median graphs in this article.}), where we give to hyperplanes the role played by edges in trees. Such a strategy turns out to be relevant in many situations. Already in \cite{MR1347406}, it is proved that a group $G$ has at least two ends relative to a subgroup $H$ if and only if its acts (essentially) on a median graph with $H$ as a hyperplane-stabiliser, generalising Stallings' theorem which claims that a group is multi-ended if and only if it acts (non-trivially) on a tree with finite edge-stabilisers. From this perspective, edge-preserving maps become \emph{parallel-preserving} maps, i.e.\ maps that sends vertices to vertices, edges to edges, and parallel edges to parallel edges (two edges being parallel whenever they are connected by a sequence of edges such that any two consecutive edges are opposite sides in a $4$-cycle).

\medskip \noindent
We begin by proving that there is a canonical  way to fold pairs of hyperplanes in contact.

\begin{thm}\label{Intro:Foldings}
Let $X$ be a countable median graph and $\mathcal{P}$ a collection of pairs of hyperplanes in contact. There exists a median graph $M$ and a parallel-preserving map $\zeta : X \to M$ such that the following holds. \\
\indent For every median graph $Y$ and every parallel-preserving map $\psi : X \to Y$ satisfying $\psi(A)=\psi(B)$ for all $\{A,B\} \in \mathcal{P}$, there exists a unique parallel-preserving map $\xi : M \to Y$ such that $\psi = \xi \circ \zeta$. 
\end{thm}

\noindent
Theorem~\ref{Intro:Foldings} characterises the pair $(Z,\zeta)$ uniquely, and the latter will be referred to as the \emph{folding of $X$ relative to $\mathcal{P}$}. A direct and explicit construction can be found in Section~\ref{section:FoldExplicit}. 

\medskip \noindent
In order to prove the theorem, we begin by describing how to fold a single pair of hyperplanes in contact. From a median graph $X$ and two hyperplanes in contact $A,B$, a natural way to fold $A$ and $B$ is to identify any two vertices of $X$ that are only separated by $A$ and $B$. See Figure~\ref{Folding}. However, such a graph may not be median. Nevertheless, a wallspace structure is preserved during the quotient, so we can define the folding of $X$ relative to $\{A,B\}$ by cubulating this wallspace. It turns out that such a folding satisfies a universal property as in Theorem~\ref{Intro:Foldings}, which motivates the idea that this is the ``correct'' way to fold a median graph. For an arbitrary collection $\mathcal{P}$ of pairs of hyperplanes in contact, we fix an enumeration of $\mathcal{P}$ and we fold the pairs successively as just described. The universal property satisfied by a single folding implies that the median graph $M$ thus obtained does not depend on our enumeration. This is the graph of Theorem~\ref{Intro:Foldings}. 
\begin{figure}[h!]
\begin{center}
\includegraphics[width=\linewidth]{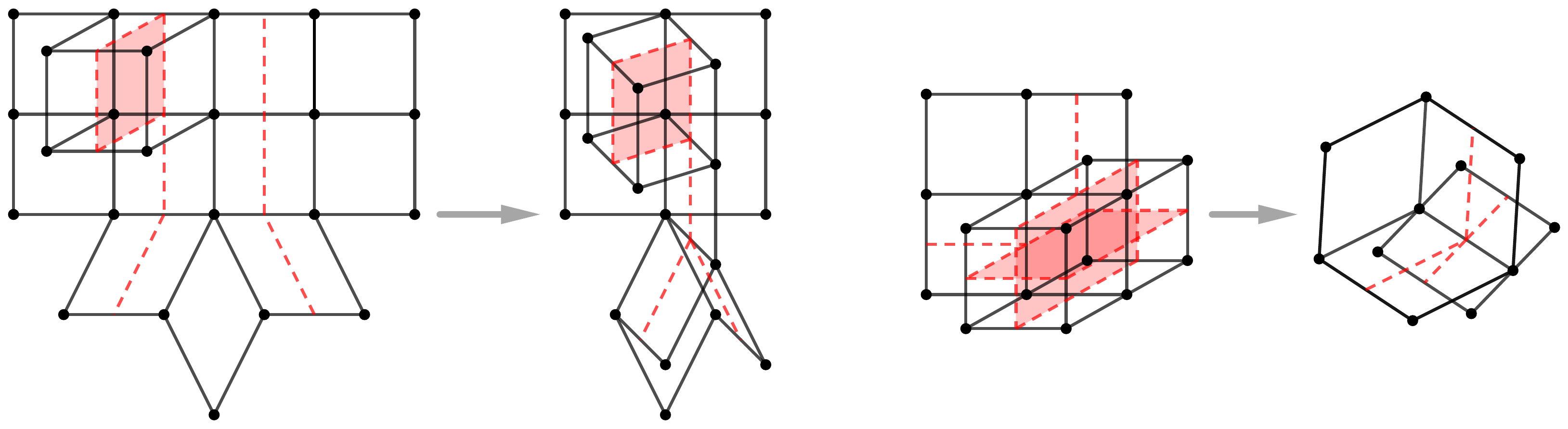}
\caption{Examples of tangent and transverse foldings.}
\label{Folding}
\end{center}
\end{figure}

\medskip \noindent
However, folding hyperplanes in contact is not sufficient in order to transform an arbitrary parallel-preserving map into an isometric embedding. The reason is that, contrary to trees, isometric embeddings between median graphs cannot be recognised locally. Another elementary transformation is required: we need to be able to make two tangent hyperplanes transverse. Again, we show that there is a canonical way to realise this construction.

\begin{thm}\label{Intro:Swelling}
Let $X$ be a countable median graph and $\mathcal{P}$ a collection of pairs of tangent hyperplanes. There exists a median graph $Z$ containing an isometrically embedded copy of $X$ such that the following holds. For every median graph $Y$, every parallel-preserving map $\psi : X \to Y$ satisfying $\psi(A) \pitchfork \psi(B)$ for every $\{A,B\} \in \mathcal{P}$ admits a unique parallel-preserving extension $Z \to Y$.\\
\indent Moreover, $Z$ is, up to isometry, the unique median graph containing an isometric copy of $X$ such that the convex hull of $X$ is $Z$ entirely and such that any two hyperplanes $A,B$ of $Z$ are transverse if and only if they are transverse in $X$ or $\{A,B\} \in \mathcal{P}$. 
\end{thm}

\noindent
The unique median graph $Z$ given by Theorem~\ref{Intro:Swelling} is referred to as the \emph{swelling of $X$ relative to $\mathcal{P}$}. A direct and explicit construction can be found in Section~\ref{section:SwellingExplicit}. 

\medskip \noindent
In order to prove the theorem, as before, we first explain how to perform our construction on a single pair of hyperplanes, i.e.\ we describe how to make a single pair of tangent hyperplanes transverse. From a median graph $X$ and two tangent hyperplanes $A,B$, a natural construction is to ``add'' the missing $4$-cycles to make $A$ and $B$ transverse. More precisely, we take the intersection $N(A) \cap N(B)$ of the carriers of $A$ and $B$, and we naturally glue a copy of $(N(A) \cap N(B)) \times [0,1]^2$ to $X$. See Figure~\ref{Swelling}. The graph thus obtained is median and satisfies again a universal property as in Theorem~\ref{Intro:Swelling}. For an arbitrary collection $\mathcal{P}$ of pairs of tangent hyperplanes, the same strategy as before works, i.e.\ we fix an enumeration of $\mathcal{P}$ and we swell the pairs successively as just described. The universal property satisfied by a single swelling implies that the median graph $Z$ thus obtained does not depend on our enumeration. This is the graph of Theorem~\ref{Intro:Swelling}. 
\begin{figure}[h!]
\begin{center}
\includegraphics[width=\linewidth]{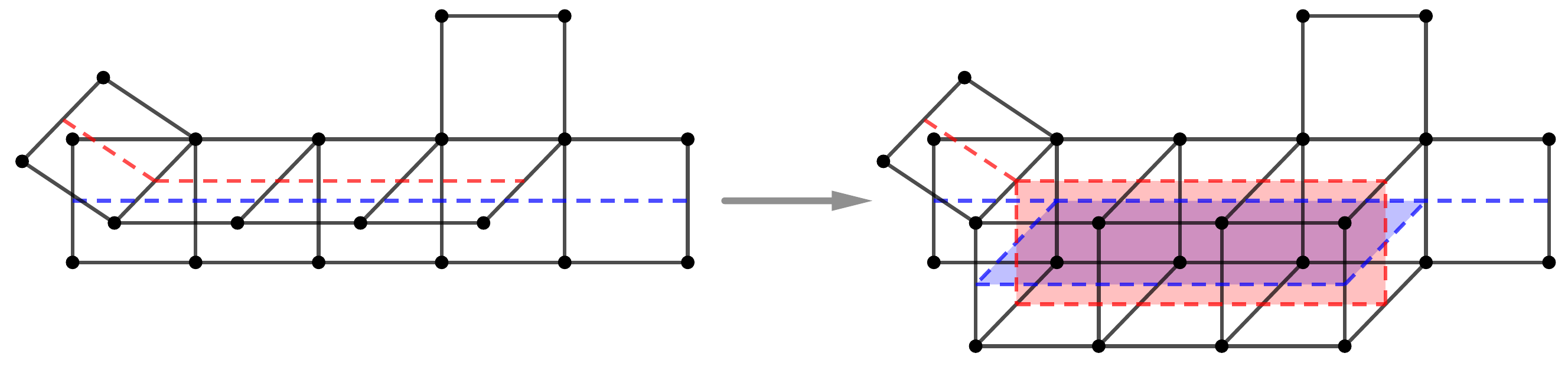}
\caption{Example of a swelling.}
\label{Swelling}
\end{center}
\end{figure}

\medskip \noindent
Finally, we show that foldings and swellings are sufficient to turn any parallel-preserving map between median graphs into an isometric embedding (with convex image).

\begin{thm}\label{Intro:mainthm}
Let $\psi : X \to Y$ be a parallel-preserving map between two median graphs $X,Y$. There exist a (possibly infinite) sequence of foldings and swellings $\eta : X \to \cdots \to Z$ and an isometric embedding (resp.\ with convex image) $\iota : Z \to Y$ such that $\psi = \iota \circ \eta$.  Moreover, $\iota(Z)$ coincides with the median (resp.\ convex) hull of $\psi(X)$ in $Y$. 
\end{thm}

\noindent
Saying that our parallel-preserving map $\psi : X \to Y$ is not an isometric embedding amounts to saying that there exist two distinct hyperplanes $A,B$ of $X$ satisfying $\psi(A)=\psi(B)$; see Lemma~\ref{lem:WhenIso}. When this happens, we would like to fold $A$ and $B$. However, $A$ and $B$ may not be in contact. Then, the trick is to apply swellings, making $A$ transverse to all the hyperplanes separating $A$ and $B$, in order to make $A$ and $B$ tangent. Thus, it is possible to construct a new median graph $X_1$ from $X$ by swelling and folding hyperplanes in which $A$ and $B$ coincide. By iterating (possibly infinitely many times) this construction, we obtain the median graph from Theorem~\ref{Intro:mainthm}. 

\medskip \noindent
Typically, the image of our isometric embedding $\iota : Z \to Y$ will be median but not convex. However, saying that $\iota(Z)$ is not convex amounts to saying that there exists two tangent hyperplanes $A$ and $B$ in $Z$ such that $\iota(A)$ and $\iota(B)$ are transverse in $Y$; see Lemma~\ref{lem:WhenIso}. Thus, we can make the image of $\iota$ convex by applying further swellings. We emphasize that a map between two median graph is an isometric embedding with convex image if and only if the map induced on the cube-completions is an isometric embedding with respect to the CAT(0) metrics. 

\medskip \noindent
It is worth mentioning that, due to the universal properties satisfied by our foldings and swellings, all our constructions can be made equivariant with respect to group actions. See Sections~\ref{section:FoldingEqui} and~\ref{section:SwellingEqui}, as well as Theorems~\ref{thm:BigFactor} and~\ref{thm:BigFactorTwo}. Under group actions, Theorem~\ref{thm:BF} provides a sufficient condition that guarantees that, in some sense, only finitely many equivariant foldings and swellings are necessary in Theorem~\ref{Intro:mainthm}.

\paragraph{Median-cocompact subgroups.} In this article, we highlight the concept of \emph{(strongly) median-cocompact subgroups}. Namely, given a group $G$ acting on a median graph $X$, a subgroup $H \leq G$ is \emph{median-cocompact} if it acts cocompactly on some connected median-closed subgraph $Y \subset X$. Here, a \emph{median-closed} subgraph $Y \subset X$ refers to a subgraph stable under the median operation, i.e.\ the median point of any three vertices of $Y$ also belongs to $Y$. Notice that, because connected median-closed subgraphs are isometrically embedded, it follows that median-cocompact subgroups in a group acting geometrically on a median graph are undistorted. 

\medskip \noindent
Median-cocompact subgroups encompass \emph{convex-cocompact} subgroups, namely subgroups acting cocompactly on convex subgraphs. They coincide with the convex-cocompact subgroups from the perspective of CAT(0) cube complexes, as subcomplexes are convex with respect to the CAT(0) metric if and only if their one-skeletons are convex with respect to the graph metric. Due to the influence of CAT(0) geometry, the literature focuses on convex-cocompact subgroups, but one can argue that the median-cocompact subgroups provide a more suited family of subgroups, which remains to be studied more deeply. Let us mention that cyclic subgroups are always median-cocompact but may not be convex-cocompact \cite{MR4645691}. Other remarkable examples include some centralisers \cite{MR4218342, MR4574362, EliaOne} and some fixators of group automorphisms \cite{EliaTwo}. Morse subgroups, including quasiconvex subgroups in (relatively) hyperbolic groups \cite{MR2413337, SageevWiseCores}, are examples of convex-cocompact subgroups \cite{MR4057355}, and a fortiori of median-cocompact subgroups. 

\medskip \noindent
As an application of our techniques, we prove the following criterion for median-cocompact subgroups:

\begin{thm}\label{thm:IntroMC}
Let $G$ be a group acting geometrically on a median graph $X$ and let $H \leq G$ be a finitely presented subgroup. Assume that one of the following conditions is satisfied:
\begin{itemize}
	\item $H$ is a locally quasiconvex hyperbolic group;
	\item $X$ has cubical dimension two.
\end{itemize}
Then $H$ is median-cocompact if and only if all the hyperplanes skewered by $H$ have finitely generated $H$-stabilisers.
\end{thm}

\noindent
Recall that a hyperbolic group is \emph{locally quasiconvex} if all its finitely generated subgroups are quasiconvex. This includes for instance free groups and hyperbolic surface groups, but also many two-dimensional hyperbolic groups.

\paragraph{Comparison with previous works.} A few works have been already dedicated to folds in median graphs from other perspectives \cite{BrownPhD, beeker2018stallings, dani2021subgroups, ben2022folding}. They will be described in details in Section~\ref{section:comparison}. In a nutshell, our goal in this paper is to decompose every parallel-preserving map between median graphs as an isometric embedding (with convex image) composed with a sequence of intuitive elementary operations that preserve the median geometry, the construction being equivariant with respect to arbitrary group actions. This contrasts with \cite{ben2022folding}, based on nonpositively curved cube complexes, which only deals with free actions and only provides isometric embeddings with convex images. And this contrasts with \cite{beeker2018stallings}, based on pocsets, which only provides isometric embeddings and which is, in our opinion, more technical and consequently less intuitive. We refer to Section~\ref{section:comparison} for a more detailed analysis.

\paragraph{Organisation of the paper.} The paper is organised as follows. In Section~\ref{section:Preliminaries}, we recall some basic properties regarding median graphs, parallel-preserving maps, and isometric embeddings between them. Theorems~\ref{Intro:Foldings} and~\ref{Intro:Swelling} are proved respectively in Sections~\ref{section:foldings} and~\ref{section:Swellings}. Section~\ref{section:isoembedding} is dedicated to Theorem~\ref{Intro:mainthm}, as well as a variation where we control the number of foldings and swellings applied; see Theorem~\ref{thm:BF}. Median-cocompact subgroups are introduced in Section~\ref{section:MedianCocompact}, where Theorem~\ref{thm:IntroMC} is proved. Finally, Section~\ref{section:comparison} is dedicated to a detailed comparison between the constructions introduced in this article and other constructions available in the literature.

\paragraph{Acknowledgements.} We are grateful to Sami Douba and Elia Fioravanti for their comments on a previous version of this article.

\section{Preliminaries}\label{section:Preliminaries}

\noindent
In this section, we record basic definitions and properties regarding median graphs and related concepts.

\paragraph{Parallelism.} First, we introduce the notion of \emph{parallel} edges in arbitrary graphs. This is especially relevant for median graphs, but parallelism will be also used for submedian graphs in Section~\ref{section:SubmedianGraph}. 

\begin{definition}
Let $X$ be a graph. Two edges $a,b \subset X$ are \emph{parallel} if there exists a sequence of edges
$$e_0=a, \ e_1, \ldots, \ e_{n-1}, \ e_n=b$$
such that $e_i,e_{i+1}$ are opposite edges in a $4$-cycle for every $0 \leq i \leq n-1$. A \emph{hyperplane} is a parallelism class of edges. Given a hyperplane $J$, we denote by $X \backslash \backslash J$ the graph obtained from $X$ by removing the (interiors of the) edges in $J$. The connected components of $X \backslash \backslash J$ are the \emph{halfspaces delimited by $J$}. The subgraph $N(J)$ induced by $J$ is the \emph{carrier of $J$}, and the connected components of $N(J) \cap X \backslash \backslash J$ are its \emph{fibers}.
\end{definition}

\noindent The following definition describes the possible intersections between hyperplanes in a graph.

\begin{definition}
Let $X$ be a graph. Two hyperplanes $J,K$ are \emph{transverse}, denoted by $J \pitchfork K$, if there exist two intersecting edges $e \in J$ and $f \in K$ that span a $4$-cycle. They are \emph{tangent} if there exist two intersecting edges $e \in J$ and $f \in K$ that do not span a $4$-cycle. Two hyperplanes that are either transverse or tangent are \emph{in contact}. 
\end{definition}

\noindent 
We now turn to the definition of maps between graphs that are relevant for us in this article, namely the \emph{parallel-preserving maps}. 

\begin{definition}
A map $\psi : X \to Y$ between two graphs $X,Y$ is \emph{parallel-preserving} if it sends vertices to vertices, edges to edges, and parallel edges to parallel edges.
\end{definition}

\noindent
Given a parallel-preserving map $\psi : X \to Y$ and a hyperplane $J$, we will denote by $\psi(J)$ the unique hyperplane of $Y$ containing the images under $\psi$ of the edges of $J$. This is well-defined because $\psi$ is parallel-preserving. However, we emphasize that not all the edges of $\psi(J)$ necessarily belong to the image of $\psi$.

\paragraph{Median geometry.} Recall that a connected graph $X$ is \emph{median} if, for all vertices $x_1,x_2,x_3 \in X$, there exists a unique vertex $m \in X$, referred to as the \emph{median}, satisfying
$$d(x_i,x_j)= d(x_i,m)+d(m,x_j) \text{ for all } i \neq j.$$
Products of trees and hypercubes (see Definition~\ref{def:Hypercube}) are typical examples of median graphs. It has been observed independently by various people that median graphs coincide with one-skeletons of CAT(0) cube complexes \cite{mediangraphs, MR1663779, Roller}. More precisely, the one-skeleton of a CAT(0) cube complex is a median graph, and, conversely, the cube complex obtained from a median graph by filling with cubes all the subgraphs isomorphic to one-skeletons of cubes is CAT(0). 

\medskip \noindent
Median graphs can be characterised by many criteria. The following one is especially useful (and essentially contained in \cite[Theorem~6.1]{mediangraphs}, even though not explicitly stated):

\begin{thm}\label{thm:MedianCriterion}
A connected graph is median if and only if
\begin{itemize}
	\item it does not contain the bipartite complete graph $K_{2,3}$ as a subgraph;
	\item it satisfies the $3$-cube condition (i.e.\ given a vertex $o$ and three neighbours $a,b,c$, if the edges $[o,a],[o,b],[o,c]$ pairwise span a $4$-cycle, then they globally span a $3$-cube);
	\item and its square-completion (i.e.\ the square complex obtained by filling with squares all the $4$-cycles) is simply connected.
\end{itemize}
\end{thm}

\noindent
Notice that the bipartite complete graph $K_{2,3}$ can be visualise as two $4$-cycles with two consecutive edges identified. In it, the two vertices of degree $3$ are median points for the three vertices of degree $2$, which is why $K_{2,3}$ is not median and more generally cannot be a subgraph of a median graph.

\medskip \noindent
Typically, the geometry of a median graph is encoded in the combinatorics of its hyperplanes. The following statement justifies this claim.

\begin{thm}[\cite{MR1347406}]\label{thm:BigMedian}
Let $X$ be a median graph. The following assertions hold.
\begin{itemize}
	\item Every hyperplane delimits exactly two halfspaces. Moreover, halfspaces, carriers, and fibers are convex.
	\item Given a hyperplane $J$ and one of its fibers $F$, there is a graph isomorphism $F \times [0,1] \to N(J)$ sending $F \times \{0\}$ to $F$ and the edges $\{\mathrm{pt}\} \times [0,1]$ to the edges of $J$. The isometry $N(J) \to N(J)$ induced by the involution of $F \times [0,1]$ defined by $(f,0) \leftrightarrow (f,1)$ is the \emph{canonical involution of $N(J)$}. 
	\item A path in $X$ is geodesic if and only if it crosses each hyperplane at most once.
	\item The distance between two vertices coincides with the number of hyperplanes separating them.
\end{itemize}
\end{thm}

\noindent 
The following lemma characterises embeddings (with convex images) between median graphs. It is the key observation that motivates our construction of folds.

\begin{lemma}\label{lem:WhenIso}
Let $X,Y$ be two median graphs and $\psi : X \to Y$ a parallel-preserving map. The following assertions hold.
\begin{itemize}
	\item $\psi$ is an isometric embedding if and only if $\psi(A) \neq \psi(B)$ for any two distinct hyperplanes $A,B$.
	\item $\psi$ is an isometric embedding with convex image if and only if, for all distinct hyperplanes $A,B$, $\psi(A) \neq \psi(B)$ and $\psi(A) \pitchfork \psi(B)$ if and only if $A \pitchfork B$.
	\item $\psi$ is an isometry if and only if $\psi$ induces a bijection from the hyperplanes of $X$ to the hyperplanes of $Y$ sending non-transverse hyperplanes to non-transverse hyperplanes.
\end{itemize}
\end{lemma}

\begin{proof}
Let $a,b \in X$ be two vertices. Let $\gamma$ be a geodesic between $a$ and $b$. Let $J_1, \ldots, J_n$ denote the hyperplanes successively crossed. Then $\psi(\gamma)$ yields a path between $\psi(a)$ and $\psi(b)$ in $Y$ crossing successively the hyperplanes $\psi(J_1), \ldots, \psi(J_n)$. If $\psi$ sends distinct hyperplanes to distinct hyperplanes, then $\psi(J_1),\ldots, \psi(J_n)$ are pairwise distinct, which implies that $\psi(\gamma)$ is a geodesic according to Theorem~\ref{thm:BigMedian}. Hence $d(\psi(a),\psi(b))=d(a,b)$. We then conclude that $\psi$ is an isometric embedding. Conversely, if $\psi$ is an isometric embedding, then $\psi(\gamma)$ must be a geodesic and the hyperplanes $\psi(J_1),\ldots, \psi(J_n)$ must be pairwise distinct. Because this has to be true for all $a,b \in X$ and that any two hyperplanes always separate two vertices, this implies that $\psi$ sends distinct hyperplanes to distinct hyperplanes. Thus, the first item of our lemma is proved.

\medskip \noindent
Now, assume that $\psi$ is an isometric embedding with convex image. We already know from the previous item that $\psi(A) \neq \psi(B)$ for any two distinct hyperplanes $A,B$. Let $A,B$ be two hyperplanes. If $A$ and $B$ are transverse, then there exists a $4$-cycle whose edges belong to $A$ and $B$. Because a parallel-preserving map has to send a $4$-cycle either to a single edge or to another $4$-cycle, it follows that either $\psi(A)=\psi(B)$ or $\psi(A) \pitchfork \psi(B)$. The former being impossible, we know that $\psi(A)$ and $\psi(B)$ must be transverse. Assuming that $A$ and $B$ are not transverse, we want to prove that $\psi(A)$ and $\psi(B)$ are not transverse either. If $N(\psi(A))$ and $N(\psi(B))$ are disjoint, there is nothing to prove. Otherwise, as a consequence of the Helly property for convex subgraphs, we know that $N(\psi(A))$ and $N(\psi(B))$ intersect in the image of $\psi$. Consequently, $N(A)$ and $N(B)$ intersect in $X$. So we can fix two intersecting edges $a \in A$ and $b \in B$ such that $\psi(a)$ and $\psi(b)$ intersect. If $\psi(a)$ and $\psi(b)$ span a $4$-cycle, then the convexity of the image of $\psi$ imposes that this $4$-cycle lies in the image of $\psi$. But $a$ and $b$ do not span a $4$-cycle, so this is not possible. Consequently, $\psi(a)$ and $\psi(b)$ do not span a $4$-cycle, which implies that $\psi(A)$ and $\psi(B)$ are not transverse.

\medskip \noindent
Conversely, assume that, for all distinct hyperplanes $A,B$, $\psi(A) \neq \psi(B)$ and $\psi(A) \pitchfork \psi(B)$ if and only if $A \pitchfork B$. We already know from the first item that $\psi$ is an isometric embedding. According to \cite{Chepoi}, in order to show that $\psi(X)$ is convex, it suffices to show that, if $a,b \subset \psi(X)$ are two intersecting edges spanning a $4$-cycle in $Y$, then they span a $4$-cycle in $\psi(X)$. Let $x,y \subset X$ be two edges such that $a=\psi(x)$ and $b=\psi(y)$. Because $a,b$ span a $4$-cycle in $Y$, the hyperplanes containing them are transverse. It follows from our assumption that the hyperplanes containing $x,y$ are transverse, which implies that $x,y$ span a $4$-cycle. Taking the image of this $4$-cycle under $\psi$ shows that $a,b$ span a $4$-cycle in $\psi(X)$, as desired. This concludes the proof of the second item of our lemma.

\medskip \noindent
Finally, assume that $\psi$ induces a bijection from the hyperplanes of $X$ to the hyperplanes of $Y$ that preserves non-transversality. Notice that $\psi$ also preserves transversality. Indeed, let $A,B$ be two transverse hyperplanes in $X$. So there exists a $4$-cycle whose edges belong to $A,B$. Because a parallel-preserving map always sends a $4$-cycle either to a single edge or to another $4$-cycle, it follows from $\psi(A) \neq \psi(B)$ that $\psi$ sends our $4$-cycle to a $4$-cycle, proving that $\psi(A)$ and $\psi(B)$ are transverse. It follows from the second item that $\psi$ is an isometric image with convex image. If $\psi$ is not surjective, then $\psi(X)$ is a proper convex subgraph of $Y$, so there must be a hyperplane of $Y$ that does not cross $\psi(X)$. But this is impossible since every hyperplane of $Y$ is the image of a hyperplane of $X$ under $\psi$. Thus, $\psi$ is surjective, proving that it defines an isometry from $X$ to $Y$. Conversely, it is clear that an isometry induces a bijection from the hyperplanes of $X$ to the hyperplanes of $Y$ that preserves non-transversality. Thus, the third item of our lemma is proved. 
\end{proof}

\noindent We now introduce a family of median graphs which will be useful in Section~\ref{section:SubmedianGraph} to define our notion of submedian graphs.

\begin{definition}\label{def:Hypercube}
Let $S$ be an arbitrary set. The \emph{hypercube} $\mathscr{HC}(S)$ is the graph whose vertices are the finite subsets of $S$ and whose edges connect two subsets with an edge whenever their symmetric difference has size one.
\end{definition}

\noindent 
Hypercubes naturally appear in the study of median graphs. In fact, median graphs can be characterised as retracts of hypercubes. We record the following weaker fact for future use:

\begin{lemma}\label{lem:MedianHC}
Median graphs embed isometrically in hypercubes.
\end{lemma}

\noindent
Lemma~\ref{lem:MedianHC} is well-known, but we include a proof as a motivation for the proof of Lemma~\ref{lem:FoldSub} later. 

\begin{proof}[Proof of Lemma~\ref{lem:MedianHC}.]
Let $S$ denote the set of the hyperplanes of $X$. Fix a basepoint $o \in X$. Because every vertex of $X$ is separated from $o$ by only finitely many hyperplanes, the map
$$\Theta : \left\{ \begin{array}{ccc} X & \to & \mathscr{HC}(S) \\ x & \mapsto & \mathcal{H}(o|x):= \{\text{hyperplanes separating $o$ and $x$}\} \end{array} \right.$$
is well-defined. Observe that, for all vertices $x,y \in X$, the symmetric difference $\Theta(x) \triangle \Theta(y)$ coincides with $\mathcal{H}(x|y)$, hence
$$d(\Theta(x), \Theta(y)) = |\Theta(x) \triangle \Theta(y)| = |\mathcal{H}(x|y)| = d(x,y)$$
according to Theorem~\ref{thm:BigMedian}. Thus, we have proved that $\Theta$ is an isometric embedding. 
\end{proof}

\paragraph{Cubulation.} Recall that a \emph{wallspace} $(X,\mathcal{W})$ is the data of a set $X$ endowed with a collection $\mathcal{W}$ of pairs of subsets such that:
\begin{itemize}
	\item for every $\{A,B\} \in \mathcal{W}$, $X= A \sqcup B$ with $A,B$ both non-empty;
	\item for all points $x,y \in X$, there exist only finitely many $\{A,B\} \in \mathcal{P}$ for which either $x \in A$ and $y \in B$ or $x \in B$ and $y \in A$.
\end{itemize}
The partitions given by $\mathcal{W}$ are \emph{walls}. The pieces of a wall are \emph{halfspaces}. A wall $\{A,B\} \in \mathcal{P}$ \emph{separates} two points $x,y \in X$ if either $x \in A$ and $y \in B$ or $x \in B$ and $y \in A$. Thus, the second item above says that any two points are separated by only finitely many walls. The \emph{wall (pseudo-)metric} is the map $X \times X \to \mathbb{N}$ that associates to two points of $X$ the number of walls separating them.

\medskip \noindent
It is explained in~\cite{MR2059193} (see also \cite{MR2197811}) how to associate a median graph to any wallspace. The construction of this median graph requires the notion of orientations.

\begin{definition}
Let $(X,\mathcal{W})$ be a wallspace. An \emph{orientation} is a map $\sigma$ that associates halfspaces to walls such that:
\begin{itemize}
	\item for every wall $W \in \mathcal{W}$, $\sigma(W) \in W$;
	\item for all $U,V \in \mathcal{W}$, $\sigma(U) \cap \sigma(V) \neq \emptyset$.
\end{itemize}
The orientation is \emph{principal} is there exists a point $x \in X$ such that, for every wall $W \in \mathcal{W}$, $\sigma(W)$ is the halfspace of $W$ containing $x$. The \emph{cubulation} of $(X,\mathcal{W})$ is the graph whose vertices are the orientations which differ from principal orientations only on finitely many hyperplanes and whose edges connect two orientations whenever they differ on a single hyperplane. 
\end{definition}

\noindent
Notice that there is a canonical map from a wallspace to its cubulation: it is the map that sends every point to the principal orientation it defines. Given an orientation $\sigma$ and a hyperplane $J$, we denote by $[\sigma,J]$ the map obtained from $\sigma$ by modifying its value on $J$. The map $[\sigma,J]$ may not be an orientation, but two orientations $\mu,\nu$ are adjacent in the cubulation if and only if $\nu = [\mu,J]$ for some hyperplane $J$. 

\medskip \noindent We now collect some fundamental properties of the cubulation of a wallspace.

\begin{prop}\label{prop:MainCubulation}
Let $(X,\mathcal{W})$ be a wallspace. Let $M(X)$ denote its cubulation and $\eta : X \to M(X)$ the canonical map. The following assertions hold. 
\begin{itemize}
	\item $M(X)$ is a median graph.
	\item The median hull of $\eta(X)$ in $M(X)$ coincides with $M(X)$. 
	\item Two edges $[a_1,b_1],[a_2,b_2] \subset M(X)$ belong to the same hyperplane if and only if $b_1=[a_1,W]$ and $b_2=[a_2,W]$ for some wall $W \in \mathcal{W}$.
\end{itemize}
\end{prop}

\noindent
Note that Proposition~\ref{prop:MainCubulation} implies the existence of a bijection between the walls in $(X,\mathcal{W})$ and the hyperplanes in $M(X)$.

\begin{proof}[Proof of Proposition~\ref{prop:MainCubulation}.]
The fact that $M(X)$ is median is proved by \cite[Proposition~4.5]{MR2059193}. We also know from \cite[Proposition~4.6]{MR2059193} that the halfspaces of $M(X)$ are exactly the subsets of the form
$$\{ \sigma \text{ orientation} \mid \sigma(W) = W_1 \},  \ W= \{W_1,W_2\} \in \mathcal{W}.$$
The third item of our proposition follows easily from this fact.

\medskip \noindent
In order to prove the second item of our proposition, we use the fact that, in a median graph, the median hull of a subgraph coincides with the intersection of all the halfspaces, \emph{$2/3$-spaces} (i.e.\ the union of two disjoint halfspaces), and \emph{$3/4$-spaces} (i.e.\ the union of two halfspaces delimited by two transverse hyperplanes) containing it. We refer the reader to \cite{Book} for more details. It follows easily from the description of halfspaces given above that every halfspace or \emph{$1/4$-space} (i.e.\ the intersection of two halfspaces delimited by two transverse hyperplanes) contains at least one principal orientation. Thus, the median hull of $\eta(X)$ must be $M(X)$ entirely.
\end{proof}

\section{Foldings}\label{section:foldings}

\noindent
In this section, we show that, given a median graph $X$ and a collection $\mathcal{P}$ of pairs of hyperplanes in contact, there exists a canonical median graph obtained from $X$ by ``merging'' or ``folding'' the pairs of hyperplanes in $\mathcal{P}$. 

\begin{thm}\label{thm:Foldings}
Let $X$ be a countable median graph and $\mathcal{P}$ a collection of pairs of hyperplanes in contact. There exists a median graph $M$ and a parallel-preserving map $\zeta : X \to M$ such that the following holds. For every median graph $Y$ and every parallel-preserving map $\psi : X \to Y$ satisfying $\psi(A)=\psi(B)$ for all $\{A,B\} \in \mathcal{P}$, there exists a unique parallel-preserving map $\xi : M \to Y$ such that $\psi = \xi \circ \zeta$. \\
\indent Moreover, for any two distinct hyperplanes $A,B$ of $X$, $\zeta(A)=\zeta(B)$ if and only if $A$ and $B$ are $\mathcal{P}$-connected. 
\end{thm}

\noindent
Here, we refer to two hyperplanes $A$ and $B$ as \emph{$\mathcal{P}$-connected} whenever there exist hyperplanes $J_0=A, J_1, \ldots, J_n=B$ such that $\{J_i,J_{i+1}\} \in \mathcal{P}$ for every $0 \leq i \leq n-1$. Being $\mathcal{P}$-connected defines an equivalence relation, and we denote by $[J]_\mathcal{P}$ or $[J]$ the equivalence class of a hyperplane $J$. 

\medskip \noindent
The universal property satisfied by the median graph $M$ implies that Theorem~\ref{thm:Foldings} uniquely characterises it, as justified in Section~\ref{section:ThmFolding}. It also implies that the construction is compatible with group actions, see Section~\ref{section:FoldingEqui}. The graph $M$ will be referred to as the \emph{fold of $X$ relative to $\mathcal{P}$}.

\subsection{Preliminaries on submedian graphs}\label{section:SubmedianGraph}

\noindent
We begin by introducing and discussing \emph{submedian graphs}. This (new) family of graphs will be fundamental in our proof of Theorem~\ref{thm:Foldings}. 

\begin{definition}
A connected graph is \emph{submedian} if it can be embedded into a hypercube and if its square-completion is simply connected.
\end{definition}

\noindent
We emphasize that a submedian graph can be embedded in a hypercube with an image that is not an induced subgraph. For instance, a $3$-cube minus (the interior of) an edge defines a submedian graph but cannot be described as an induced subgraph of a hypercube. Median graphs are examples of submedian graphs; see the discussion about median graphs in Section~\ref{section:Preliminaries}.

\medskip \noindent
A key property of submedian graphs is that they are naturally endowed with a wallspace structure, as justified by the following observation:

\begin{lemma}\label{lem:HypSeparates}
Let $X$ be a connected graph. If the square-completion of $X$ is simply connected, then hyperplanes separate.
\end{lemma}

\begin{proof}
Let $[a,b]$ be an edge in $J$. We claim that $J$ separates $a$ and $b$. Let $\gamma$ be an arbitrary path connecting $a$ and $b$. Because the square-completion $X^\square$ of $X$ is simply connected, there exists a disc diagram $\Delta \to X^\square$ bounded by $\gamma \cup [a,b]$, i.e.\ a contractible planar square complex $\Delta$ endowed with a combinatorial map $\Delta \to X^\square$ whose restriction to $\partial \Delta$ coincides with $\gamma \cup [a,b]$ (see e.g.\ \cite{MR0675732}). In $\Delta$, it is clear that there must be an edge from $\gamma$ parallel to $[a,b]$, hence the desired conclusion. 
\end{proof}

\noindent
Next, let us mention a characterisation of submedian graphs.

\begin{prop}\label{prop:SubCriterion}
A connected graph is submedian if and only if its square-completion is simply connected and every path connecting two distinct vertices crosses at least one parallelism class an odd number of times.
\end{prop}

\begin{proof}
Let $X$ be a connected graph. First, assume that $X$ is submedian and embed $X$ into a hypercube $Q$. Given two distinct vertices $a,b \in X$ and a path $\gamma \subset X$ connecting them, if $J$ denotes a hyperplane of $Q$ separating $a$ and $b$ then $\gamma$ has to cross $J$ an odd number of times. However, two edges in $X$ may be parallel in $Q$ but not in $X$. Nevertheless, it is clear that two edges parallel in $X$ are parallel in $Q$, so $J \cap X$ is a union of hyperplanes of $X$. Necessarily, $\gamma$ has to cross at least one of them an odd number of times. 

\medskip \noindent
Conversely, assume that $X$ satisfies the two conditions given by our proposition. It is proved in \cite{MR1402650} that $X$ embeds into a hypercube if and only if there is colouring $\kappa : E(X) \to C$ of its edges satisfying the following conditions:
\begin{itemize}
	\item given a colour $c \in C$, every loop has an even number of edges with colour $c$;
	\item for every path $\gamma$ of positive length, there exists a colour $c \in C$ such that $\gamma$ has an odd number of edges with colour $c$.
\end{itemize}
Set $C := \{ \text{parallelism classes of } X\}$ and let $\kappa : E(X) \to C$ denote the map that sends each edge to the hyperplane containing it. The first item above is satisfied as a consequence of Lemma~\ref{lem:HypSeparates} and the second item is satisfied by assumption. We conclude that $X$ is submedian. 
\end{proof}

\noindent
We record the following elementary observation for future use.

\begin{lemma}\label{lem:Wheel}
Let $X,Y$ be two submedian graphs. A parallel-preserving map $\psi : X \to Y$ sends every $4$-cycle either to a single edge or to another $4$-cycle.
\end{lemma}

\begin{proof}
Let $(a,b,c,d)$ be the four vertices of a $4$-cycle $C$. If $\psi$ is injective on $C$, there is nothing to prove. Otherwise, $\psi$ maps two vertices of $C$ to the same vertex of $Y$. Since $\psi$ sends edges to edges, two such vertices cannot be adjacent, so they must be opposite vertices. Up to relabelling the vertices of $C$, say that $\psi(a)= \psi(c)$. Because $\psi$ is parallel-preserving, the edges $\psi([a,b])$, $\psi([b,c])$, and $\psi([c,d])=\psi([a,d])$ must be parallel. But, in a submedian graph, two distinct parallel edges cannot intersect (because this cannot happen in a hypercube), so $\psi$ must collapse $C$ to a single edge. 
\end{proof}

\begin{cor}\label{cor:Wheel}
Let $X,Y$ be two submedian graphs. A parallel-preserving map $\psi : X \to Y$ sends every corner (i.e.\ the data of a vertex $o$ and three neighbours $a,b,c$ such that the edges $[o,a],[o,b],[o,c]$ pairwise span $4$-cycles) either to a single edge or to another corner.
\end{cor}

\begin{proof}
If $\psi$ is injective on our corner $W$, then there is nothing to prove. Otherwise, if $\psi$ is not injective, two cases may happen. First, $\psi$ may be injective on each $4$-cycle of our corner. But then $\psi$ sends the corner to a subgraph containing the complete bipartite graph $K_{2,3}$, which is impossible since $Y$ is submedian.  Second, $\psi$ may not be injective on at least one $4$-cycle of our corner. It follows from Lemma~\ref{lem:Wheel} that at least one of the three $4$-cycles is collapsed to a single edge. If there is only one such $4$-cycle, then $\psi$ sends $W$ to a copy of the complete bipartite graph $K_{2,3}$ in $Y$, which is impossible since $Y$ is submedian. If there is at least two such $4$-cycles, then $\psi$ must collapse $W$ to a single edge. 
\end{proof}

\noindent
The rest of the section is dedicated to the proof of the following proposition. Roughly speaking, it states that cubulating the wallspace given by the hyperplanes of a submedian graph yields a canonical procedure that associates a median graph to a submedian graph. 

\begin{prop}\label{prop:CubulatingSub}
Let $X$ be a submedian graph, $Y$ a median graph, and $\psi : X \to Y$ a parallel-preserving map. There exists a unique parallel-preserving map $\xi : M(X) \to Y$ such that $\psi = \xi \circ \eta$, where $\eta : X \to M(X)$ denotes the canonical embedding into the cubulation $M(X)$ of $X$. 
\end{prop}

\noindent
We begin by recording some observations about cubulations of submedian graphs. 

\begin{lemma}\label{lem:SubCubulation}
Let $X$ be a submedian graph and let $M(X)$ denote its cubulation. 
\begin{itemize}
	\item The canonical map $\eta : X \to M(X)$ is a graph embedding. In particular, it is parallel-preserving.
	\item Two edges in $\eta(X)$ are parallel in $\eta(X)$ if and only if they are parallel in $M(X)$. 
\end{itemize}
\end{lemma}

\begin{proof}
Let $x,y \in X$ be two distinct vertices and let $\mu,\nu$ denote the principal orientations they respectively define. As a consequence of Proposition~\ref{prop:SubCriterion}, there exists a hyperplane $J$ crossed an odd number of times by some path connecting $x$ to $y$. Necessarily, $J$ separates $x$ and $y$, hence $\mu(J) \neq \nu(J)$, and a fortiori $\mu \neq \nu$. Hence $\eta(x) \neq \eta(y)$. Moreover, if $x$ and $y$ are adjacent in $X$, then the hyperplane containing the edge $[x,y]$ is the unique hyperplane separating $x$ and $y$. Thus, $\mu$ and $\nu$ differ on a single hyperplane, which implies that they are adjacent in $M(X)$. We conclude that $\eta$ is a graph embedding.

\medskip \noindent
Let $e_1,e_2$ be two edges in $X$. It is clear that, if $\eta(e_1)$ and $\eta(e_2)$ are parallel in $\eta(X)$, then they are parallel in $M(X)$ as well. Conversely, assume that $\eta(e_1)$ and $\eta(e_2)$ are parallel in $M(X)$. Write $e_1:=[a_1,b_1]$ and $e_2:=[a_2,b_2]$. According to Proposition~\ref{prop:MainCubulation}, there exists a hyperplane $J$ in $X$ such that $\eta(b_1)=[\eta(a_1),J]$ (resp. $\eta(b_2)= [\eta(a_2),J]$). In other words, $J$ is the unique hyperplane separating $a_1$ and $b_1$ (resp. $a_2$ and $b_2$), which amounts to saying that $e_1$ and $e_2$ both belong to $J$. Thus, $e_1$ and $e_2$ are parallel in $X$, which amounts to saying that $\eta(e_1)$ and $\eta(e_2)$ are parallel in $\eta(X)$ since $\eta$ is a graph embedding. 
\end{proof}

\noindent
In order to prove Proposition~\ref{prop:CubulatingSub}, we need the following statement, which shows how one can construct the cubulation of a submedian graph step by step.

\begin{lemma}\label{lem:CubulationInductively}
Let $X$ be a submedian graph. Let $M(X)$ denote its cubulation and $\eta : X \to M(X)$ the canonical embedding. There exists an increasing sequence $(M_\alpha)_{\alpha \leq \Omega}$ of subgraphs of $M(X)$, indexed by some ordinal $\Omega$, connecting $M_0:=\eta(X)$ to $M_\Omega:=M(X)$ and satisfying:
\begin{itemize}
	\item for every ordinal $\alpha< \Omega$, $M_{\alpha+1}$ is a submedian graph obtained from $M_\alpha$ by completing a corner as a $3$-cube with an existing or unexisting vertex;
	\item for every limit ordinal $\alpha \leq \Omega$, $M_\alpha$ is a submedian graph obtained as the union of the $M_\beta$ for $\beta<\alpha$;
	\item for every ordinal $\alpha \leq \Omega$, every edge in $M_\alpha$ is parallel to an edge in $M_0$;
	\item for every $\alpha \leq \Omega$, two edges in $M_\alpha$ are parallel in $M_\alpha$ if and only if they are parallel in $M(X)$.
\end{itemize}
\end{lemma}

\begin{proof}
We construct the sequence of subgraphs by transfinite induction. First, we set $M_0:=\eta(X)$. This is a submedian graph as a consequence of Lemma~\ref{lem:SubCubulation}.

\medskip \noindent
Next, assume that $M_\alpha$ is already constructed for some ordinal $\alpha$. If $M_\alpha=M(X)$, we can stop the construction. Otherwise, if $M_\alpha \subsetneq M(X)$, then it follows from Proposition~\ref{prop:MainCubulation} that $M_\alpha$ is not a median graph. Because, as a submedian graph, it does not contain $K_{2,3}$ as a subgraph and its square-completion is simply connected, it follows from Theorem~\ref{thm:MedianCriterion} that $M_\alpha$ contains a corner $W$ that does not span a $3$-cube in $M_\alpha$. But, since $M(X)$ is median, $W$ can be completed in $M(X)$ with a vertex which may or may not belong to $M_\alpha$. Then, we define $M_{\alpha+1}$ as the subgraph obtained from $M_\alpha$ by completing $W$. It is clear that $M_{\alpha+1}$ is a subgraph of a hypercube, since it is a subgraph of the median graph $M(X)$. Moreover, every loop in the square-completion of $M_{\alpha+1}$ can be easily homotoped in the square-completion of $M_\alpha$, so $M_{\alpha+1}$ must have a simply connected square-completion. Thus, $M_{\alpha+1}$ is a submedian graph. It is also clear that every edge of $M_{\alpha+1}$ not in $M_{\alpha}$ is parallel to an edge in $M_{\alpha}$, so a fortiori it has to be parallel to an edge in $M_0$.

\medskip \noindent
Finally, assume that there exists some limit ordinal $\lambda$ such that $M_\alpha$ is already constructed for $\alpha< \lambda$. Then define $M_\lambda$ as the union of the $M_\alpha$ for $\alpha<\lambda$. It is clear that $M_\lambda$ is a submedian subgraph and that every edge in $M_\lambda$ is parallel to some edge in $M_0$, because these properties already hold for all the $M_\alpha$ with $\alpha<\lambda$. 

\medskip \noindent
Of course, the process has to stop at some ordinal $\Omega$ whose cardinality is bounded above by the sum of the cardinalities of the vertex- and edge-sets of $M(X)$. So far, we have constructed our sequence of subgraphs and we have verified that the first three items of our lemma hold. Let us verify the fourth item in order to conclude the proof of our lemma. 

\medskip \noindent
Let $a,b \subset M_\alpha$ be two edges parallel in $M(X)$. It follows from what we already know that there exists an edge $a' \subset M_0$ (resp.\ $b' \subset M_0$) such that $a'$ (resp.\ $b'$) is parallel to $a$ (resp. $b$). According to Lemma~\ref{lem:SubCubulation}, $a'$ and $b'$, which are parallel in $M(X)$, must be parallel in $\eta(X)$ as well, and a fortiori in $M_\alpha$ since $\eta(X) = M_0 \subset M_\alpha$. Thus, two edges in $M_\alpha$ parallel in $M(X)$ are also parallel in $M_\alpha$. The converse is clear. 
\end{proof}

\begin{proof}[Proof of Proposition~\ref{prop:CubulatingSub}.]
Let $(M_\alpha)_{\alpha \leq \Omega}$ be the sequence given by Lemma~\ref{lem:CubulationInductively}. For every $\alpha \leq \Omega$, we claim that there exists a unique parallel-preserving map $\xi_\alpha : M_\alpha \to Y$ satisfying $\psi = \xi_\alpha \circ \eta$. We argue by transfinite induction over $\alpha$. 

\medskip \noindent
For $\alpha=0$, we set $\xi_0:= \psi \circ \eta^{-1}|_{\eta(X)}$ and there is nothing to prove. Next, assume that $\lambda\leq \Omega$ is a limit ordinal such that our claim holds for every $\alpha < \lambda$. Because $\xi_\mu$ and $\xi_\nu$ must agree on $M_\alpha$ for all $\mu,\nu \geq \alpha$, and since $M_\lambda$ is the union of the $M_\alpha$ for $\alpha< \lambda$, we can (and must) define $\xi_\lambda$ as $x \mapsto \xi_\alpha(x)$ if $\alpha$ is such that $x \in M_\alpha$. Because the $\xi_\alpha$, $\alpha<\lambda$, are all parallel-preserving, so is $\xi_\lambda$. 

\medskip \noindent
Now, assume that our claim holds for some $\alpha < \Omega$. We know that $M_{\alpha+1}$ can be obtained from $M_\alpha$ by completing a corner $W \subset M_\alpha$. Let $x \in W$ denote the centre of $W$, $x_1,x_2,x_3 \in W$ its neighbours, and, for all $1 \leq i < j \leq 3$, let $x_{ij} \in W$ denote the fourth vertex of the $4$-cycle spanned by $x,x_i,x_j$. We distinguish two cases, depending on whether $W$ is completed with a vertex $p$ in $M_\alpha$ or not. 

\medskip \noindent
First, assume that $p$ does not belong to $M_\alpha$. According to Corollary~\ref{cor:Wheel}, a parallel-preserving map sends a corner either to a single edge or to another corner. Consequently, since $p$ also defines a corner with $x_1,x_2,x_3,x_{12},x_{13},x_{23}$, either $\xi_\alpha$ sends $W$ to a single edge  and we must set $\xi_{\alpha+1}(p)=\xi_\alpha(x_1)$ (or equivalently, $\xi_\alpha(x_2)$ or $\xi_\alpha(x_3)$); or $\xi_\alpha$ sends $W$ to another corner $W'$ in $Y$ and we must define $\xi_{\alpha+1}(p)$ as the unique vertex of $Y$ completing $W'$ to a $3$-cube. In both cases, it is clear that $\xi_{\alpha+1}$ sends vertices to vertices and edges to edges. It remains to show that $\xi_{\alpha+1}$ is parallel-preserving. Let $e_1,e_2 \subset M_{\alpha+1}$ be two distinct parallel edges. If neither $e_1$ nor $e_2$ contains $p$, then $e_1,e_2 \subset M_\alpha$ and we already know that $\xi_{\alpha+1}(e_1)= \xi_\alpha(e_1)$ and $\xi_{\alpha+1}(e_2)= \xi_\alpha(e_2)$ are parallel. Notice that, because two distinct edges in a submedian graph cannot be parallel if they intersect, $e_1$ and $e_2$ cannot both contain $p$. Therefore, it only remains to consider the case where $p \notin e_1$ and $p \in e_2$ (up to switching $e_1$ and $e_2$). The endpoints of $e_2$ are then $p$ and $x_{ij}$ for some $1 \leq i<j \leq 3$, so $e_2$ is parallel to $e_2':= [x_i,x_{ik}]$ where $k \neq i,j$. On the one hand, $\xi_{\alpha+1}(e_1)=\xi_{\alpha}(e_1)$ and $\xi_{\alpha+1}(e_2')= \xi_\alpha(e_2')$ are parallel because $\xi_\alpha$ is parallel-preserving; and, on the other hand, $\xi_{\alpha+1}(e_2)$ is parallel to $\xi_\alpha(e_2')$ by construction. Consequently, $\xi_{\alpha+1}(e_1)$ and $\xi_{\alpha+1}(e_2)$ are parallel in $Y$, concluding the proof that $\xi_{\alpha+1}$ is parallel-preserving. 

\medskip \noindent
Next, assume that $p$ belongs to $M_\alpha$. In order to define $\xi_{\alpha+1} : M_{\alpha+1} \to Y$, it suffices to verify that, for all $1 \leq i<j \leq 3$, $\xi_\alpha(p)$ and $\xi_{\alpha}(x_{ij})$ are adjacent in $Y$. Because $M_\alpha$ is connected, there exists a path $\gamma \subset M_\alpha$ connecting $p$ to $x_{ij}$. The fact that $p$ and $x_{ij}$ are adjacent in $M(X)$ implies that the partition of the edges of $\gamma$ in parallelism classes in $M(X)$ contains exactly one piece of odd cardinality (given by the parallelism class of $[x_i,x_{ik}]$ where $k \neq i,j$). It follows from Lemma~\ref{lem:CubulationInductively} that the same partition is obtained from parallelism classes in $M_\alpha$. Because $\xi_\alpha$ is parallel-preserving, it sends a parallelism class inside a parallelism class, with the possibility that two distinct parallelism classes are sent inside the same parallelism class in $Y$. Consequently, the partition of $\xi_\alpha(\gamma)$ in parallelism classes in $Y$ contains exactly one piece of odd cardinality (given by the parallelism class of $[\xi_\alpha(x_i),\xi_\alpha(x_{ik})]$), which implies that the endpoints of $\xi_\alpha(\gamma)$, namely $\xi_\alpha(p)$ and $\xi_\alpha(x_{ij})$, are adjacent in $Y$. Let us record our conclusion:

\begin{fact}\label{fact:Parallel2}
The map $\xi_\alpha$ sends the vertices $p$ and $x_{ij}$ to two adjacent vertices of $Y$. Moreover, this edge is parallel to $[\xi_\alpha(x_i),\xi_\alpha(x_{ik})]$. 
\end{fact}

\noindent
Thus, $\xi_\alpha$ extends uniquely to a map $\xi_{\alpha+1} : M_{\alpha+1} \to Y$ that sends vertices to vertices and edges to edges. It remains to verify that $\xi_{\alpha+1}$ is parallel-preserving. This can done almost word for word as before, but we repeat the argument for clarity. Let $e_1,e_2 \subset M_{\alpha+1}$ be two distinct parallel edges. If neither $e_1$ nor $e_2$ contains $p$, then $e_1,e_2 \subset M_\alpha$ and we already know that $\xi_{\alpha+1}(e_1)= \xi_\alpha(e_1)$ and $\xi_{\alpha+1}(e_2)= \xi_\alpha(e_2)$ are parallel. Notice that, because two distinct edges in a submedian graph cannot be parallel if they intersect, $e_1$ and $e_2$ cannot both contain $p$. Therefore, it only remains to consider the case where $p \notin e_1$ and $p \in e_2$ (up to switching $e_1$ and $e_2$). The endpoints of $e_2$ are then $p$ and $x_{ij}$ for some $1 \leq i<j \leq 3$, so $e_2$ is parallel to $e_2':= [x_i,x_{ik}]$ where $k \neq i,j$. On the one hand, $\xi_{\alpha+1}(e_1)=\xi_{\alpha}(e_1)$ and $\xi_{\alpha+1}(e_2')= \xi_\alpha(e_2')$ are parallel because $\xi_\alpha$ is parallel-preserving; and, on the other hand, $\xi_{\alpha+1}(e_2)$ is parallel to $\xi_\alpha(e_2')$ according to Fact~\ref{fact:Parallel2}. Consequently, $\xi_{\alpha+1}(e_1)$ and $\xi_{\alpha+1}(e_2)$ are parallel in $Y$, concluding the proof that $\xi_{\alpha+1}$ is parallel-preserving. 
\end{proof}

\subsection{Folding two hyperplanes}\label{section:FoldingTwo}

\noindent
Now we are ready to fold hyperplanes in median graphs. There is a naive construction that merges two hyperplanes:

\begin{definition}
Let $X$ be a median graph and $A,B$ two hyperplanes in contact. Let $\alpha$ (resp.\ $\beta$) denote the canonical involution of $N(A)$ (resp.\ $N(B)$). The \emph{first fold of $X$ relative to $\{A,B\}$} is the graph obtained from $X$ by identifying $\alpha(p)$ and $\beta(p)$ for every $p \in N(A) \cap N(B)$. 
\end{definition}

\noindent
We can distinguish two types of folds, depending on whether the two hyperplanes $A,B$ are tangent or transverse. In the latter case, folding amounts to collapsing all the $4$-cycles in $N(A) \cap N(B)$ to edges. See Figure~\ref{Folding} for examples. Unfortunately, the first fold of a median graph may not be median. Nevertheless, it turns out to be submedian. In particular, it keeps a wallspace structure. 

\begin{lemma}\label{lem:FoldSub}
Let $X$ be a median graph and $A,B$ two hyperplanes in contact. The first fold $Z$ of $X$ relative to $\{A,B\}$ is submedian. Moreover, the quotient map $X \twoheadrightarrow Z$ is parallel-preserving.
\end{lemma}

\noindent
Our proof of the statement is inspired by the proof of Lemma~\ref{lem:MedianHC} given above.

\begin{proof}[Proof of Lemma~\ref{lem:FoldSub}.]
For all vertices $x,y \in X$, define
$$\delta(x,y) := \left\{ \begin{array}{cl} d(x,y)-2 & \text{if both $A,B$ separate $x,y$} \\ d(x,y) & \text{otherwise} \end{array} \right..$$
Clearly, for all distinct $x,y \in X$, $\delta(x,y) \geq 0$ with equality if and only if $\mathcal{H}(x|y)=\{A,B\}$, which amounts to saying that there exists some $p \in N(A) \cap N(B)$ such that $x=\alpha(p)$ and $y=\beta(p)$. Moreover, $\delta(x,y)=1$ if and only if either $x,y$ are adjacent in $X$ or $\mathcal{H}(x|y)=\{A,B,J\}$ for some hyperplane $J \neq A,B$. In the latter case, there exists a vertex $z \in X$ such that, up to switching $x$ and $y$, we have $x$ adjacent to $z$ and $\delta(z,y)=0$. Indeed, fix an arbitrary geodesic $\gamma$ between $x$ and $y$. The hyperplanes crossing $\gamma$ must be exactly $A,B,J$. If $\gamma$ contains two consecutive edges crossed by $A,B$, then our claim follows. Otherwise, $J$ crosses the middle edge of $\gamma$ and $A,B$ cross the first and third edges of $\gamma$. Because $A$ and $B$ are in contact, $J$ cannot separate $A$ and $B$, so it has to be transverse to $A$ or $B$. This implies that the middle edge of $\gamma$ spans a $4$-cycle with the first or third edge. Consequently, we can modify $\gamma$ in order to get two consecutive edges crossed by $A$ and $B$, concluding the proof of our claim. 

\medskip \noindent
From the previous argument, it follows that $\delta$ defines a pseudo-metric on $X$ whose corresponding metric space is isometric to our folding $Z$. Fix a basepoint $o \in X$ and define
$$\Theta : \left\{ \begin{array}{ccc} X & \to & \mathscr{HC} \left( \mathcal{H}(X) \backslash \{A,B\} \sqcup \{Q\} \right) \\ x & \mapsto & \left\{ \begin{array}{cl} \mathcal{H}(o|x) & \text{if neither $A$ nor $B$ separate $o,x$} \\ \mathcal{H}(o|x)\backslash \{A\} \sqcup \{Q\} & \text{if $A$ separates $o,x$ but not $B$} \\ \mathcal{H}(o|x) \backslash \{B\} \sqcup \{Q\} & \text{if $B$ separates $o,x$ but not $A$} \\ \mathcal{H}(o|x) \backslash \{A,B\} & \text{if $A,B$ both separate $o,x$} \end{array} \right. \end{array} \right..$$
In other words, $\Theta$ sends each vertex $x$ to the (finite) collection of the hyperplanes separating $o$ and $x$, with the convention that $A$ and $B$ are considered as a single hyperplane $A \sqcup B$ which we denote by $Q$. By construction, we have the symmetric difference
$$\Theta(x) \Delta \Theta(y) = \left\{ \begin{array}{cl} \mathcal{H}(x|y) & \text{if neither $A$ nor $B$ separate $x,y$} \\ \mathcal{H}(x|y)\backslash \{A\} \sqcup \{Q\} & \text{if $A$ separates $x,y$ but not $B$} \\ \mathcal{H}(x|y) \backslash \{B\} \sqcup \{Q\} & \text{if $B$ separates $x,y$ but not $A$} \\ \mathcal{H}(x|y) \backslash \{A,B\} & \text{if $A,B$ both separate $x,y$} \end{array} \right.$$
for all vertices $x,y \in X$. It follows that $\Theta$ induces an isometric embedding from $(X,\delta)$ into a hypercube, proving that our folding $Z$ is a partial cube.

\medskip \noindent
It is clear that the quotient map $\pi : X \twoheadrightarrow Z$ sends vertices to vertices and edges to edges. Moreover, a $4$-cycle is sent to a single edge if the hyperplanes crossing it are exactly $A$ and $B$; and it is sent to a $4$-cycle otherwise. This implies that $\pi$ is parallel-preserving.

\medskip \noindent
It remains to show that the square completion of $Z$ is simply connected. So let $\gamma$ be a loop in $Z$. A priori, $\gamma$ is not the image of a loop in $X$ under $\pi$. However, it can be described as the image of a \emph{broken loop} $\zeta$ in $X$, i.e.\ a cycle of vertices such that any two consecutive vertices either are adjacent in $X$ or have the same projection to $Z$. If $a,b$ are two consecutive vertices of $\zeta$ having the same projection to $Z$, we can connect them with a path of length two whose projection to $Z$ is a backtrack. Thus, we can construct a (non-broken) loop $\zeta^+$ in $X$ whose projection to $Z$ differ from the projection of $\zeta$ only by addition of backtracks, and so defines a loop in $Z$ homotopically equivalent to $\gamma$. Since $\zeta^+$ becomes homotopically trivial in the square-completion of $X$ and because $\pi$ sends $4$-cycles to edges or $4$-cycles, we conclude that $\gamma$ is homotopically trivial in the square-completion of $Z$. Thus, $Z$ is a submedian graph. 
\end{proof}

\begin{remark}
It is worth mentioning that our proof of Lemma~\ref{lem:FoldSub} shows that the first fold of a median graph isometrically embeds into a hypercube. As a consequence, it can be deduced that the first fold of a median graph turns out to be semi-median (as introduced in \cite{MR1642702}). However, we emphasize that iterating first folds may produce graphs that are not semi-median, nor even submedian. See Figure~\ref{Bad}. 
\end{remark}
\begin{figure}[h!]
\begin{center}
\includegraphics[width=0.8\linewidth]{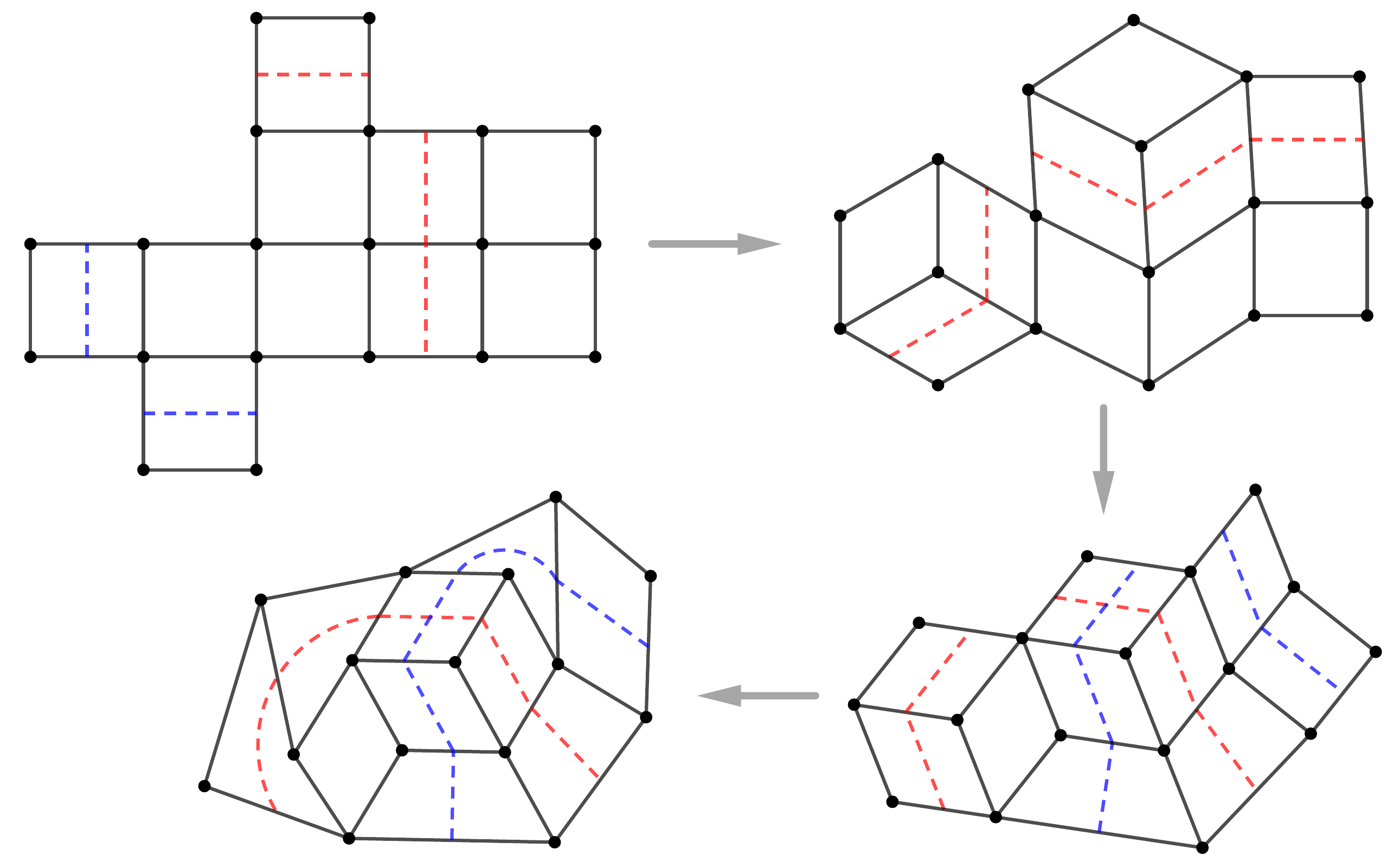}
\caption{A graph that is not submedian but obtained from a median graph by a sequence of foldings.}
\label{Bad}
\end{center}
\end{figure}

\noindent
First folds are the natural operations that merge pairs of hyperplanes. However, as already mentioned, a first fold of a median graph may not be median. However, since the fold turns out to be submedian according to Lemma~\ref{lem:FoldSub}, it admits a natural wallspace structure which we can cubulate in order to get a median graph. 

\begin{definition}
Let $X$ be a median graph and $A,B$ two hyperplanes in contact. The \emph{fold of $X$ relative to $\{A,B\}$} is the cubulation of the first fold of $X$ relative to $\{A,B\}$.
\end{definition}

\noindent
The next proposition justifies that our definition of fold is the most natural one.

\begin{prop}\label{prop:FoldingTwoHyp}
Let $X$ be a median graph and $A,B$ two hyperplanes in contact. Let $Z$ denote the fold of $X$ relative to $\{A,B\}$ and $\eta : X \to Z$ the canonical map. For every median graph $Y$ and every parallel-preserving map $\psi : X \to Y$ satisfying $\psi(A)=\psi(B)$, there exists a unique parallel-preserving map $\xi : Z \to Y$ such that $\psi = \xi \circ \eta$. \\
\indent Moreover, for any two distinct hyperplanes $J,K$ of $X$, $\eta(J)=\eta(K)$ if and only if $\{J,K\} = \{A,B\}$. 
\end{prop}

\begin{proof}
Let $\alpha$ (resp.\ $\beta$) denote the canonical involution of $N(A)$ (resp.\ $N(B)$). Given a vertex $p \in N(A) \cap N(B)$, the path $[\alpha(p),p] \cup [p, \beta(p)]$ is sent by $\psi$ to a path of length two that crosses the hyperplane $\psi(A)=\psi(B)$ twice. This implies that $\psi(\alpha(p))=\psi(\beta(p))$. Consequently, $\psi$ uniquely factors through the first fold $X'$ of $X$ relative to $\{A,B\}$, i.e.\ there exists a map $\zeta : X' \to Y$ such that $\psi = \zeta \circ \pi$ where $\pi : X \to X'$ denotes the canonical projection. More precisely, $\zeta$ is defined by taking, given a vertex $a \in X'$, an arbitrary pre-image $x \in \pi^{-1}(a)$ and by setting $\zeta(z):= \psi(x)$. Our previous observation precisely shows that the vertex of $Y$ thus obtained does not depend on the choice of the pre-image.

\medskip \noindent
Because every edge of $X'$ is the image of an edge of $X$ and that $\psi$ sends an edge to an edge, we know that $\zeta$ sends edges to edges. It remains to show that $\zeta$ is parallel-preserving. 

\begin{claim}\label{claim:PiHyperplanes}
For all hyperplanes $J,K$ of $X$, $\pi(J)=\pi(K)$ if and only if $J=K$ or $\{J,K\} = \{A,B\}$. 
\end{claim}

\noindent
We claim that, if two edges of $X$ are sent by $\pi$ to two opposite edges in some $4$-cycle of $X'$, then our two initial edges necessarily belong to the same hyperplane $X$ or belong to $A$ and $B$. This is sufficient to deduce our statement since, given two hyperplanes $J,K$ of $X$ with the same image in $X'$, we can take a sequence of edges $\pi(e_1), \ldots, \pi(e_n)$ such that $e_1$ belongs to $J$, $e_n$ belongs to $K$, and $\pi(e_i)$ is opposite to $\pi(e_{i+1})$ in some $4$-cycle for every $1 \leq i \leq n-1$. Then our claim implies that, for every $1 \leq i \leq n-1$, $e_i$ and $e_{i+1}$ either belong to the same hyperplane or they belong to $A$ and $B$. This amounts to saying that either $J=K$ or $\{J,K\}= \{A,B\}$.

\medskip \noindent
So let $[a,b]$ and $[x,y]$ be two edges such that $\pi(a)$ is adjacent to $\pi(x)$ and such that $\pi(b)$ is adjacent to $\pi(y)$. Because $d(a,x)$ and $d(b,y)$ can be equal either to $1$ or to $3$, there are different cases to distinguish. If $d(a,x)$ and $d(b,y)$ are both equal to $1$, then $a,b,x,y$ define a $4$-cycle and the edges $[a,b]$ and $[x,y]$ clearly belong to the same hyperplane. If exactly one of $d(a,x)$ and $d(b,y)$ is equal to $1$, say $d(a,x)=1$ and $d(b,y)=3$, then there are hyperplanes $J$ and $K \in \{A,B\}$ such that $\mathcal{H}(a|x)=\{J\}$ and $\mathcal{H}(b|y)=\{A,B,K\}$. If $K$ separates $a$ and $b$, then $\mathcal{H}(a|y)=\{A,B\}$, which implies that $\pi(a)=\pi(y)$, which is impossible. Similarly, $K$ cannot separate $x$ and $y$. Consequently, $K=J$ and our edges $[a,b]$ and $[x,y]$ belong to $A$ and $B$. Finally, if $d(a,x)$ and $(b,y)$ are both equal to $3$, then there exist hyperplanes $J,K$ such that $\mathcal{H}(a|x)=\{A,B,J\}$ and $\mathcal{H}(b|y) = \{A,B,K\}$. If $J$ separates $a$ and $b$, then $\mathcal{H}(b|y)=\{A,B\}$, hence $\pi(b)=\pi(y)$, which is not possible. Similarly, $J$ cannot separate $x$ and $y$, which implies that $J$ must separate $b$ and $y$, i.e.\ $J=K$. It follows that the hyperplane separating $a$ and $b$ coincides with the hyperplane separating $x$ and $y$. This concludes the proof of Claim~\ref{claim:PiHyperplanes}. 

\medskip \noindent
Now, let us prove that $\zeta$ is parallel-preserving. Let $e,f \subset X'$ be two parallel edges. Fix two edges $e',f' \subset X$ respectively sent to $e,f$ by $\pi$. As a consequence of Claim~\ref{claim:PiHyperplanes}, either $e'$ and $f'$ are parallel, which implies that $\xi(e)=\psi(e')$ and $\xi(f')= \psi(f)$ are parallel since $\psi$ is parallel-preserving; or $e'$ and $f'$ belong to $A$ and $B$, which implies that $\xi(e)=\psi(e')$ and $\xi(f')= \psi(f)$ are parallel since $\psi(A)= \psi(B)$.

\medskip \noindent
So far, we have proved that there exists a unique parallel-preserving map $\zeta : X' \to Y$ satisfying $\psi = \zeta \circ \pi$. Next, as a consequence of Proposition~\ref{prop:CubulatingSub} (which applies according to Lemma~\ref{lem:FoldSub}) to the cubulation $Z:=M(X')$, we know that there exists a unique parallel-preserving map $\xi : Z \to Y$ satisfying $\zeta = \xi \circ \eta$, where $\eta : X' \to Z$ is the canonical embedding. This concludes the proof of the first assertion of our proposition. The second assertion follows from Claim~\ref{claim:PiHyperplanes} and Proposition~\ref{prop:MainCubulation}.
\end{proof}

\noindent
It is worth noticing that, as justified by the next example, in general there is no canonical way to fold two hyperplanes that are not in contact.

\begin{remark}\label{remark:NotCanonicalFolding}
Let $X$ be a path of length four, and let $A,B,C,D$ denote its successive hyperplanes. In order to fold $A$ and $D$, we have two natural constructions. See Figure~\ref{NoFold}. First, we can fold $B$ and $C$, in order to make $A$ and $D$ tangent, and then fold $A$ and $D$. The median graph $Y$ thus obtained is a path of length two. Next, we can make $A$ transverse to both $B$ and $C$, in order to make $A$ and $D$ tangent, and then fold $A$ and $D$. The median graph $Z$ thus obtained is two $4$-cycles with a common edge. Unfortunately, there is no median graph $W$ in which $A,D$ coincide so that $X \to Y$ and $X \to Z$ factor through $W$ via parallel-preserving maps. 

\begin{figure}[h!]
\begin{center}
\includegraphics[width=0.6\linewidth]{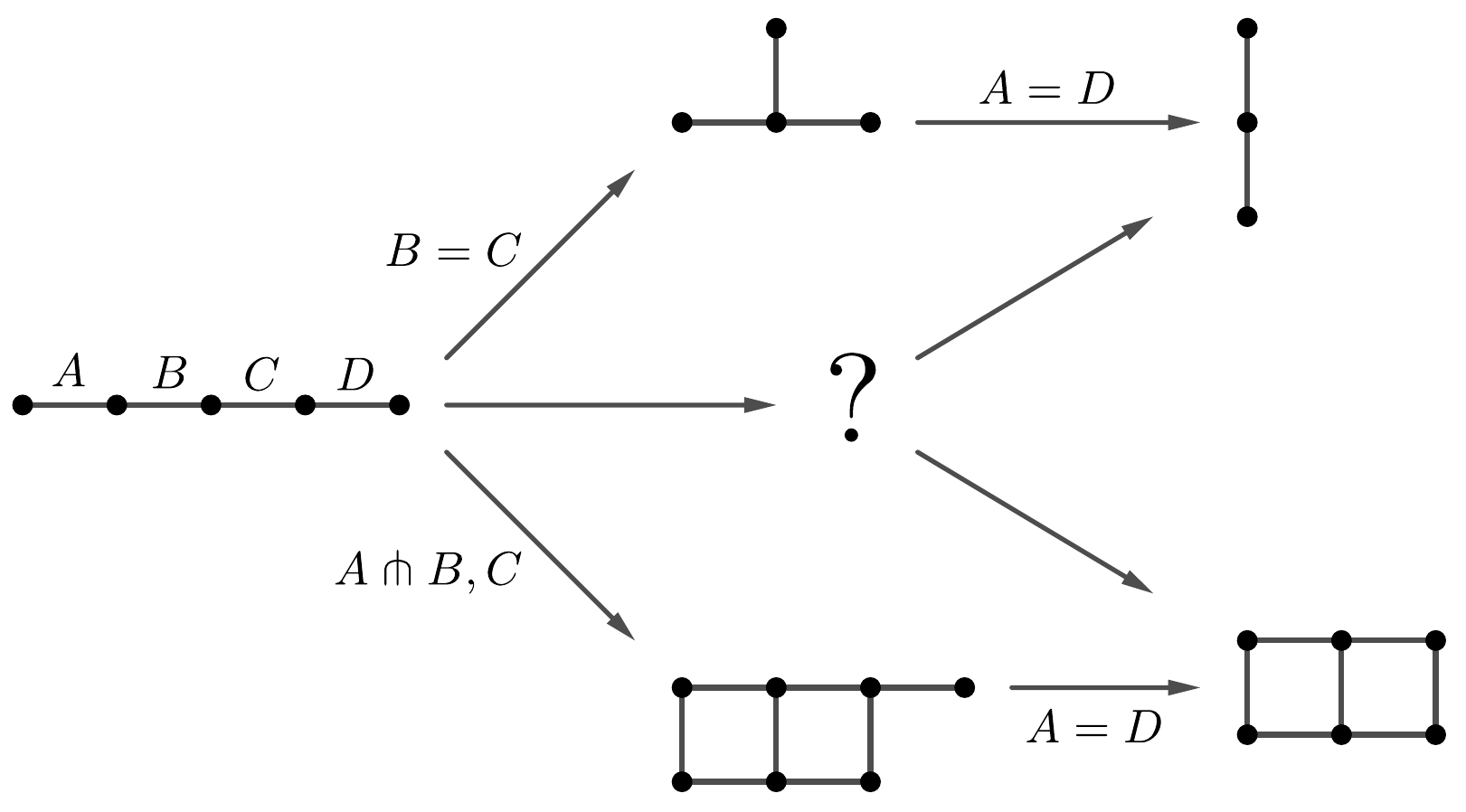}
\caption{A configuration where there is no canonical way to fold two hyperplanes.}
\label{NoFold}
\end{center}
\end{figure}
\end{remark}

\subsection{Folding hyperplanes}\label{section:ThmFolding}

\noindent
So far, we have seen how to fold a pair of hyperplanes in contact. Now, given an arbitrary collection of pairs of hyperplanes in contact, we prove Theorem~\ref{thm:Foldings} by folding the pairs successively. A priori, the median graph thus constructed depends on the enumeration we use to fold the pairs of hyperplanes. However, it will follow from the universal property provided by Proposition~\ref{prop:FoldingTwoHyp} that one obtains the same median graph independently of the enumeration, which is illustrated by the universal property claimed by Theorem~\ref{thm:Foldings}.

\medskip \noindent
In order to define properly the graph obtained by applying infinitely many foldings and swellings, let us recall that:

\begin{definition}\label{def:DirectLimit}
Let $(X_i)_{i \geq 0}$ be a sequence of graphs and $(\sigma_i : X_i \to X_{i+1})_{i \geq 0}$ a sequence of maps (sending vertices to vertices and edges either to edges or to vertices). The \emph{direct limit} $\lim\limits_{\longrightarrow} (X_i,\sigma_i)_{i \geq 0}$ is the graph
\begin{itemize}
	\item whose vertex-set is $\bigsqcup_{i \geq 0} X_i$ quotiented by the equivalence relation that identifies two vertices $a \in X_i$ and $b \in X_j$ whenever they have the same image in some $X_k$, i.e.\ if there exist $r,s \geq 0$ such that $\sigma_{i+r} \circ \cdots \circ \sigma_i (a)= \sigma_{j+s} \circ \cdots \circ \sigma_j(b)$;
	\item and whose edges connect two vertices $a \in X_i$ and $b \in X_j$ (say with $i \leq j$) whenever their images in $X_k$ are eventually adjacent, i.e.\ if $\sigma_{j+s} \circ \cdots \circ \sigma_j \circ \sigma_{j-1} \circ \cdots \circ \sigma_i(a)$ and $\sigma_{j+s} \circ \cdots \circ \sigma_j(b)$ are adjacent in $X_{j+s+1}$ for every $s \geq 1$ large enough.
\end{itemize}
\end{definition}

\noindent
Notice that each $X_i$ maps canonically to the direct limit. 

\begin{proof}[Proof of Theorem~\ref{thm:Foldings}.]
Up to adding redundancy in $\mathcal{P}$ and saying that folding a pair of hyperplanes given by a single hyperplane does not modify the median graph, we can assume that $\mathcal{P}$ is infinite. Fix an enumeration $\mathcal{P}= \{ \{A_i,B_i\}, i \geq 1\}$. We construct median graphs $Z_i$, parallel-preserving maps $\eta_i : X \twoheadrightarrow Z_i$, and parallel-preserving maps $\pi_{i+1} : Z_i \to Z_{i+1}$ for every $i \geq 0$ by the following inductive process:
\begin{itemize}
	\item We set $Z_0:=X$ and $\eta_0= \mathrm{id}_X$. 
	\item At a step $i \geq 1$, we define $Z_{i+1}$ as the folding of $Z_i$ relative to $\{A_{i+1},B_{i+1}\}$, we let $\pi_{i+1} : Z_i \twoheadrightarrow Z_{i+1}$ be the canonical projection, and we set $\eta_{i+1}:= \pi_{i+1} \circ \eta_i$. 
\end{itemize}
Now, we can define $Z$ as the direct limit of the $(Z_i,\pi_{i+1})$ and $\eta : X \to Z$ as the limit of the $\eta_i$.

\medskip \noindent
Observe that, because the $\pi_i$ cannot increase distances (indeed, they send edges to edges), the distance in $Z$ between two vertices in $Z_i$ and $Z_j$ agrees with the distance between their images in $Z_k$ for $k$ large enough. It easily follows that $Z$ must be a median graph. 

\medskip \noindent
Now, we want to show that $Z$ and $\eta$ satisfy the universal property given by the theorem. So let $Y$ be a median graph and $\psi : X \to Y$ a parallel-preserving map satisfying $\psi(A)=\psi(B)$ for every $\{A,B\} \in \mathcal{P}$. According to Proposition~\ref{prop:FoldingTwoHyp}, there exist maps $\xi_i : Z_i \to Y$ such that, for every $i \geq 1$, $\xi_i$ is the unique parallel-preserving map $Z_i \to Y$ satisfying $\xi_{i-1}= \xi_i \circ \pi_i$. Then the limit $\xi$ of the $\xi_i$ yields a commutative diagram

\medskip
\hspace{3cm}
\xymatrix{
X \ar[rrrr]^\psi \ar[drr]^{\eta_0} \ar[ddrr]^{\eta_1} \ar[ddddrr]^\eta && && Y \\ && Z_0 \ar[urr]^{\xi_0} \ar[d]^{\pi_1} && \\ && Z_1 \ar[uurr]^{\xi_1} \ar[d]^{\pi_2} && \\ && \vdots \ar[d] && \\ && Z \ar[uuuurr]^\xi && 
}

\medskip \noindent
It is clear that $\xi$ is parallel-preserving and that $\psi = \xi \circ \eta$. The fact that there exists a unique such map $\xi$ follows from the observation that $\psi = \eta_i \circ (\pi_{i+1} \circ \pi_{i+2} \circ \cdots \circ \xi)$ for every $i \geq 0$ and from the uniqueness provided by Proposition~\ref{prop:FoldingTwoHyp}.

\medskip \noindent
Finally, observe that, as a consequence of Proposition~\ref{prop:FoldingTwoHyp}, we know that, for every $i \geq 0$, two distinct hyperplanes $J,K$ of $X$ have the same image under $\eta_i$ if and and only if their images under $\eta_{i-1}$ coincide or yields the pair $\{A_i,B_i\}$. The second assertion of our theorem follows by induction.
\end{proof}

\noindent
It is worth noticing that, given a countable median graph $X$ and a collection $\mathcal{P}$ of pairs of hyperplanes in contact, if $(M_1,\zeta_1)$ and $(M_2,\zeta_2)$ both satisfy the conclusion of Theorem~\ref{thm:Foldings}, then $M_1$ and $M_2$ are naturally isomorphic. 

\medskip \noindent
Indeed, because $\zeta_1(A)=\zeta_1(B)$ for all hyperplanes $\{A,B\} \in \mathcal{P}$, there must exist a unique parallel-preserving map $\xi_2 : M_2 \to M_1$ such that $\zeta_1 = \xi_2 \circ \zeta_2$. But we know that, for any two distinct hyperplanes $A,B$ of $X$, $\zeta_1(A)=\zeta_1(B)$ if and only if $A$ and $B$ are $\mathcal{P}$-connected, which also amounts to saying that $\zeta_2(A)=\zeta_2(B)$. Consequently, $\xi_2$ has to induce a bijection from the hyperplanes of $M_2$ to the hyperplanes of $M_1$. Similarly, there exists a unique parallel-preserving map $\xi_1 : M_1 \to M_2$ such that $\zeta_2 = \xi_1 \circ \zeta_1$, and $\xi_1$ induces a bijection from the hyperplanes of $M_2$ to the hyperplanes of $M_1$. From
$$\zeta_1 = \xi_2 \circ \zeta_2 = \xi_2 \circ\xi_1 \circ \zeta_1,$$
it follows that $\zeta_2 = \xi_1 \circ \xi_2 \circ \xi_1 \circ \zeta_1$, and the uniqueness of $\xi_1$ implies that $\xi_1 = \xi_1 \circ \xi_2 \circ \xi_1$. As a consequence, the permutation on the hyperplanes of $M_2$ induced by $\xi_1 \circ \xi_2$ must be the identity. Let $A,B$ be two hyperplanes with $\xi_2 (A)$ and $\xi_2(B)$ are transverse. Notice that, because $\xi_1$ sends injectively the hyperplanes of $M_1$ to the hyperplanes of $M_2$, it follows from Lemma~\ref{lem:Wheel} that $\xi_1$ preserves transversality. Consequently, $\xi_1 \circ \xi_2(A) = A$ and $\xi_1 \circ \xi_2 (B)=B$ must be transverse. Thus, we have proved that $\xi_2$ preserves non-transversality. We conclude from Lemma~\ref{lem:WhenIso} that $\xi_2 : M_2 \to M_1$ is an isometry.

\medskip \noindent
Our previous observation justifies the following definition:

\begin{definition}
Let $X$ be a countable median graph and $\mathcal{P}$ a collection of pairs of hyperplanes in contact. The median graph given by Theorem~\ref{thm:Foldings} and its map $\eta : X \to Z$ is the \emph{folding of $X$ relative to $\mathcal{P}$}. 
\end{definition}

\noindent
A direct description will be given in Section~\ref{section:FoldExplicit}.

\subsection{Equivariance}\label{section:FoldingEqui}

\noindent
Typically, there will be groups acting on our median graphs, and we would like our foldings to be equivariant under the actions. Our next proposition shows that the universal property satisfied by foldings automatically implies that our constructions are equivariant under group actions. 

\begin{prop}\label{prop:equivariancefolding}
Let $G$ be a group acting on a median graph $X$ and $\mathcal{P}$ a $G$-invariant collection of pairs of hyperplanes in contact. Let $Z$ denote the folding of $X$ relative to $\mathcal{P}$ and $\eta : X \to Z$ the associated projection. The following assertions hold.
\begin{itemize}
	\item For every isometry $\alpha$ of $X$ preserving $\mathcal{P}$, there exists a unique isometry $\beta$ of $Z$ such that $\beta \circ \eta = \eta \circ \alpha$. Consequently, there is a canonical action of $G$ on $Z$ making $\eta$ $G$-equivariant. 
	\item Let $H$ be a group acting on a median graph $Y$ and $\varphi : G \to H$ a morphism. For every $\varphi$-equivariant parallel-preserving map $\psi : X \to Y$ satisfying $\psi(A)=\psi(B)$ for every $\{A,B\} \in \mathcal{P}$, there exists a unique $\varphi$-equivariant parallel-preserving map $\xi : Z \to Y$ such that $\psi = \xi \circ \eta$. 
\end{itemize}
\end{prop}

\begin{proof}
By applying the universal property of $(Z,\eta)$ to $\eta \circ \alpha$ (resp. $\eta \circ \alpha^{-1}$), we find a parallel-preserving map $\beta : Z \to Z$ (resp. $\gamma : Z \to Z$) such that $\eta \circ \alpha= \beta \circ \eta$ (resp. $\eta \circ \alpha^{-1}= \gamma \circ \eta$). We have
$$\eta = \eta \circ \alpha \circ \alpha^{-1}= \beta \circ \eta \circ \alpha^{-1}= \beta \circ \gamma \circ \eta$$
and
$$\eta = \eta \circ \alpha^{-1} \circ \alpha = \gamma \circ \eta \circ \alpha = \gamma \circ \beta \circ \eta.$$
The uniqueness in the universal property satisfied by $(Z,\eta)$ imposes that $\beta \circ \gamma= \mathrm{id}_Z$ and $\gamma \circ \beta = \mathrm{id}_Z$. Thus, $\beta$ is the graph isomorphism we are looking for. This proves the first item of our proposition.

\medskip \noindent
In order to prove the second item, we need to verify that the unique parallel-preserving map $\xi : Z \to Y$ given by the universal property of $(Z,\eta)$ is $\varphi$-equivariant. Notice that, for every $g \in G$, the map $\zeta_g : Z \to Y$ defined by $z \mapsto \varphi(g)^{-1} \cdot \xi(g \cdot z)$ satisfies $\psi = \zeta_g \circ \eta$. Indeed, for every $x \in X$, we have
$$\zeta_g (\eta(x)) = \varphi(g)^{-1} \cdot \xi(g \cdot \eta(x)) = \varphi(g)^{-1} \cdot \xi ( \eta(g \cdot x)) = \varphi(g)^{-1} \cdot \psi(g \cdot x) = \psi(x)$$
because $\psi$ is $\varphi$-equivariant. The uniqueness of $\xi$ implies that $\zeta_g=\xi$, which amounts to saying that $\xi(g \cdot z ) = \varphi(g) \cdot \xi(z)$ for every $z \in Z$. Thus, $\xi$ is $\varphi$-equivariant. 
\end{proof}

\subsection{An explicit description}\label{section:FoldExplicit}

\noindent
Given a countable median graph $X$ and a collection $\mathcal{P}$ of pairs of tangent hyperplanes, we saw in the proof of Theorem~\ref{thm:Foldings} that the folding $(M,\zeta)$ of $X$ relative to $\mathcal{P}$ can be obtained by folding the pairs in $\mathcal{P}$ successively, as described in Section~\ref{section:FoldingTwo}. In this last subsection, our goal is to provide a direct description of $M$. This will be useful in Section~\ref{section:Freemediancocompact}. 

\medskip \noindent
For every hyperplane $J$ of $X$, define a wall $\mathcal{W}(J)$ on $X$ as follows: take the union of all the hyperplanes in the class $[J]$ of $J$ (with respect to the relation of being $\mathcal{P}$-connected), and say that two vertices of $X$ belong to the same halfspace delimited by $\mathcal{W}(J)$ if they are separated by an even number of hyperplanes in $[J]$. Set $\mathcal{W}:= \{ \mathcal{W}(J) , \text{ $J$ hyperplane}\}$. 

\begin{prop}
Let $Z$ denote the cubulation of the wallspace $(X,\mathcal{W})$. There exists an isometry $\phi : M \to Z$ such that $\pi = \phi \circ \zeta$, where $\pi : X \to Z$ is the canonical map.
\end{prop}

\begin{proof}
It is clear that the canonical map $\pi : X \to Z$ is parallel-preserving. Because $\pi(A)=\pi(B)$ for all $\{A,B\} \in \mathcal{P}$, Theorem~\ref{thm:Foldings} implies that there exists a parallel-preserving map $\phi : M \to Z$ satisfying $\pi = \phi \circ \zeta$. Because there is a natural bijection preserving transversality and non-transversality between the hyperplanes of $Z$ and the walls of $\mathcal{W}$, it follows from the second part of Theorem~\ref{thm:Foldings} and from the definition of $\mathcal{W}$ that $\phi$ induces a bijection between the sets of hyperplanes of $M$ and $Z$. We conclude from Lemma~\ref{lem:WhenIso} that $\phi$ is an isometry.
\end{proof}

\section{Swellings}\label{section:Swellings}

\noindent
After ``merging'' hyperplanes in median graphs, we would like to be able to make tangent hyperplanes transverse. The following theorem records the main properties satisfied by the construction explained in the next subsections.

\begin{thm}\label{thm:Swelling}
Let $X$ be a countable median graph and $\mathcal{P}$ a collection of pairs of tangent hyperplanes. There exists a median graph $Z$ containing an isometrically embedded copy of $X$ such that the following holds. For every median graph $Y$, every parallel-preserving map $\psi : X \to Y$ satisfying $\psi(A) \pitchfork \psi(B)$ for every $\{A,B\} \in \mathcal{P}$ admits a unique parallel-preserving extension $Z \to Y$.\\
\indent Moreover, $Z$ is, up to isometry, the unique median graph containing an isometric copy of $X$ such that the convex hull of $X$ is $Z$ entirely and such that any two hyperplanes $A,B$ of $Z$ are transverse if and only if they are transverse in $X$ or $\{A,B\} \in \mathcal{P}$. 
\end{thm}

\noindent
The universal property satisfied by the median graph $Z$ implies that Theorem~\ref{thm:Swelling} uniquely characterises it, as justified in Section~\ref{section:Sw}. It also implies that the construction is compatible with group actions, see Section~\ref{section:SwellingEqui}. The graph $Z$ will be referred to as the \emph{swelling of $X$ relative to $\mathcal{P}$}.

\subsection{Swelling two hyperplanes}\label{section:SwellingTwo}

\noindent
We begin by explaining how to make two tangent hyperplanes transverse. 

\begin{definition}
Let $X$ be a median graph and $A,B$ two tangent hyperplanes. Let $\alpha$ (resp. $\beta$) denote the canonical involution of $N(A)$ (resp. $N(B)$). The \emph{swelling of $X$ relative to $\{A,B\}$} is the graph obtained from the union\footnote{Here, $[0,1]$ is thought of as the complete graph $K_2$ on $\{0,1\}$. This implies that $[0,1]^2$ is a $4$-cycle.}
$$X \sqcup (N(A) \cap N(B)) \times [0,1]^2$$
by identifying $N(A) \cap N(B) \times \{(0,0)\}$ with $N(A) \cap N(B)$, $N(A) \cap N(B) \times \{(1,0)\}$ with $\alpha(N(A) \cap N(B))$, and $N(A) \cap N(B) \times \{(0,1)\}$ with $\beta(N(A) \cap N(B))$. 
\end{definition}

\noindent
See Figure~\ref{Swelling} for an example. The next proposition records the key properties satisfied by our construction. 

\begin{prop}\label{prop:SwellingTwo}
Let $X$ be a median graph and $A,B$ two tangent hyperplanes. The swelling $Z$ of $X$ relative to $\{A,B\}$ is a median graph containing $X$ isometrically. \\
\indent Moreover, the convex hull of $X$ in $Z$ is $Z$ entirely; and any two hyperplanes $J,H$ of $Z$ are transverse if and only if they are transverse in $X$ or $\{J,H\}=\{A,B\}$.
\end{prop}

\begin{proof}
Define a \emph{spot} as a map $\sigma$ satisfying the following conditions:
\begin{itemize}
	\item $\sigma$ sends each hyperplane of $X$ to one of the halfspaces it delimits;
	\item $\sigma(J) \cap \sigma(H) \neq \emptyset$ for all hyperplanes $\{J,H\} \neq \{A,B\}$;
	\item $\sigma$ differs only on finitely many hyperplanes with some (or equivalently, any) principal orientation of $X$.
\end{itemize}
Let $Z'$ denote the graph whose vertices are the spots and whose edges connect two spots whenever they differ on a single hyperplane. Observe that $Z'$ contains a natural copy of $X$ given by the subgraph spanned by the principal orientations of $X$. 

\begin{claim}\label{claim:SpotVersion}
The canonical embedding $X \hookrightarrow Z'$ extends to an isomorphism $Z \to Z'$.
\end{claim}

\noindent
For every $p \in N(A) \cap N(B)$, let $\gamma(p)$ denote the vertex $\{p\} \times \{(1,1)\}$ of $Z$. Let $\theta_p$ denote the map obtained from the principal orientation given by $p$ by modifying the values it takes at $A$ and $B$. Because there is no hyperplane separating $p$ from $A$ nor $B$, $\theta_p$ is a spot. We want to prove that the map $\Theta : Z \to Z'$ that sends a vertex of $X$ to its principal orientation and each $\gamma(p)$ to $\theta_p$ induces a graph isomorphism. We begin by proving that $\Theta$ induces a graph embedding. By construction, it is clear that $\Theta$ induces a graph isomorphism $X \to X$. Moreover, two vertices $x \in X$ and $p \in Z\backslash X$ cannot have the same image under $\Theta$ since $\Theta(x)(A) \cap \Theta(x)(B) \neq \emptyset$ whereas $\Theta(p)(A) \cap \Theta(p)(B)= \emptyset$. Therefore, it suffices to show that $\Theta$ is injective on $Z \backslash X$ and that it preserves adjacency and non-adjacency between vertices in $Z \backslash X$ and between vertices of $X$ and $Z \backslash X$.
\begin{itemize}
	\item Let $p,q \in N(A) \cap N(B)$ be two vertices. By construction, the spots $\theta_p$ and $\theta_q$ differ exactly on the hyperplanes separating $p$ and $q$. Therefore, $\Theta(\gamma(p)) = \theta_p \neq \theta_q= \Theta(\gamma(q))$ are distinct whenever $p$ and $q$ are distinct (i.e.\ separated by at least one hyperplane). Also, $\Theta(\gamma(p))=\theta_p$ and $\Theta(\gamma(q))= \theta_q$ are adjacent in $Z'$ if and only if $p$ and $q$ are adjacent in $X$ (i.e.\ separated by exactly one hyperplane), which amounts to saying that $\gamma(p)$ and $\gamma(q)$ are adjacent in $Z$. 
	\item Let $p \in N(A) \cap N(B)$ and $x \in X$ be two vertices. By construction, the set of the hyperplanes on which $\Theta(x)$ and $\Theta(\gamma(p))=\theta_p$ differ is the symmetric difference $\mathcal{H}(x|p) \triangle \{A,B\}$. Notice that $A$ and $B$ cannot both belong to $\mathcal{H}(x|p)$ because they are tangent in $X$. Consequently, $\Theta(x)$ and $\theta_p$ are adjacent in $Z'$ if and only if $x=\alpha(p)$ or $x=\beta(p)$, which amounts to saying that $x$ and $p$ are adjacent in $Z$. 
\end{itemize}
In order to conclude the proof our claim, it remains to verify that $\Theta$ is surjective. So let $\sigma$ be a spot. If $\sigma$ is an orientation of $X$, which amounts to saying that $\sigma(A) \cap \sigma(B) \neq \emptyset$, then $\sigma$ must be a principal orientation and a fortiori it belongs to the image of $\Theta$. Otherwise, if $\sigma(A) \cap \sigma(B)= \emptyset$, let $\varsigma$ denote the map obtained from $\sigma$ by modifying its values on $A$ and $B$. Because $\sigma$ is a spot and because there is no hyperplane separating $A$ and $B$, we know that $\varsigma$ is an orientation. Let $p \in X$ denote the vertex given by the principal orientation $\varsigma$. If $p \notin N(A)$, then there must exist some hyperplane $J$ separating $p$ from $N(A)$. But then either $\sigma(J) \cap \sigma(A) = \varsigma(J) \cap \varsigma(A)^c = \emptyset$ if $J \neq B$, which contradicts the fact that $\sigma$ is a spot; or $\varsigma(B) \subset \varsigma(A)$ if $J=B$, which implies that $\sigma(B) \cap \sigma(A) = \varsigma(B)^c \cap \varsigma(A)^c = \varsigma(A)^c \neq \emptyset$ and contradicts the fact that $\sigma$ is not an orientation. Thus, the vertex $p$ must belong to $N(A)$. Similarly, it has to belong to $N(B)$ as well. We conclude that $\Theta(\gamma(p)) = \sigma$. Claim~\ref{claim:SpotVersion} is thus proved. 

\begin{claim}\label{claim:DistanceSpot}
The distance between two spots in $Z'$ coincides with the number of hyperplanes on which they differ.
\end{claim}

\noindent
Let $n$ be the number of hyperplanes on which two fixed spots $\sigma$ and $\sigma'$ differ. We prove our claim by induction over $n$. The case $n \leq 1$ is immediate, so from now on we assume that $n \geq 2$. Let $\mathcal{J}=\{J_1,\ldots,J_n\}$ be the set of the hyperplanes on which $\sigma$ and $\sigma'$ differ. Up to reordering, we can suppose that $\sigma'(J_1)$ is maximal with respect to the inclusion in the set $\{\sigma'(J_1),\ldots,\sigma'(J_n)\}$. Let $\eta$ denote the map that coincides with $\sigma$ for all the hyperplanes of $X$ but $J_1$ and that sends $J_1$ to $\sigma'(J_1)$. 

\medskip \noindent
We claim that $\eta$ is a spot. Indeed, let $H_1,H_2$ be a pair of hyperplanes of $X$ such that $\{H_1,H_2\} \neq \{A,B\}$. It suffices to prove that $\eta(H_1) \cap \eta(H_2) \neq \emptyset$. 
\begin{itemize}
	\item If $H_1,H_2 \neq J_1$, then $\eta(H_1) \cap \eta(H_2)=\sigma(H_1) \cap \sigma(H_2) \neq \emptyset$ since $\sigma$ is a vertex of $X'$. So we may suppose that $H_1=J_1$. 
	\item If $H_2 \notin \mathcal{J}$, then $\eta(H_1) \cap \eta(H_2)=\sigma'(H_1) \cap \sigma'(H_2) \neq \emptyset$. Thus, we may also suppose that $H_2 \in \mathcal{J}$. 
	\item Assume for contradiction that $\eta(H_1) \cap \eta(H_2)=\emptyset$. This implies that $H_2 \neq J_1$, so that $\eta(H_2)=\sigma(H_2)$. Since $H_2 \in \mathcal{J}$, we see that $\sigma(H_2)$ is the complement of $\sigma'(H_2)$. Thus, $\sigma'(J_1) \subseteq \sigma'(H_2)$, which contradicts the maximality of $\sigma'(J_1)$. 
\end{itemize}
Thus, we have proved that $\eta$ is a spot.

\medskip \noindent
By construction, the vertices $\sigma$ and $\eta$ are adjacent in $Z'$. By induction, we have $d_{Z'}(\eta,\sigma')=n-1$. Thus, we have $d_{Z'}(\sigma,\sigma') \leq n$. By definition of $Z'$, we have $n \leq d_{Z'}(\sigma,\sigma')$. Therefore, we have $d_{Z'}(\sigma,\sigma') = n$, as desired. Claim~\ref{claim:DistanceSpot} is proved.

\medskip \noindent
It follows immediately from Claim~\ref{claim:DistanceSpot} that $X$ is isometrically embedded in $Z'$. In order to conclude the proof of the first assertion in our proposition, we show that:

\begin{claim}\label{claim:ZMedian}
The graph $Z'$ is median.
\end{claim}

\noindent
Let $\sigma_1,\sigma_2,\sigma_3$ be three vertices in $Z'$. Let  $\sigma$ denote the map that associates to a hyperplane $H$ of $X$ the halfspace of $X$ appearing at least twice in $(\sigma_1(H),\sigma_2(H),\sigma_3(H))$. The construction of $\sigma$ implies that it defines a spot. Moreover, it follows from Claim~\ref{claim:DistanceSpot} that, for all pairwise distinct $i,j,k \in \{1,2,3\}$, 
$$
\begin{array}{ccl}
  d_{Z'}(\sigma_i,\sigma_j) &=&  \#\{H \text{ hyperplane} \mid \sigma_i(H) \neq \sigma_j(H)\}\\
     {} &=& \#\{H \text{ hyperplane} \mid \sigma_i(H) \neq \sigma_j(H)=\sigma_k(H)\} \\
     {} & {} & +\#\{H \text{ hyperplane} \mid \sigma_k(H)=\sigma_i(H) \neq \sigma_j(H)\} \\
     {} &=&  d_{Z'}(\sigma_i,\sigma)+d_{Z'}(\sigma,\sigma_j).
\end{array}
$$
Thus, the vertex $\sigma$ is a median point for the triple $\{\sigma_1,\sigma_2,\sigma_3\}$. This proves existence.

\medskip \noindent
We now prove uniqueness of median points. Let $\sigma' \in Z'$ be a median point of the triple $\{\sigma_1,\sigma_2,\sigma_3\}$. We prove that $\sigma'=\sigma$. Let $i,j \in \{1,2,3\}$ be distinct and let $H$ be a hyperplane such that $\sigma_i(H)=\sigma_j(H)$. Notice that Claim~\ref{claim:DistanceSpot} implies that following a geodesic from $\sigma_i$ to $\sigma_j$ amounts to modifying successively the values taken by $\sigma_i$ on the hyperplanes on which $\sigma_i$ and $\sigma_j$ differ, and only on these hyperplanes. Therefore, since $\sigma'$ lies on a geodesic between $\sigma_i$ and $\sigma_j$, it follows that $\sigma'(H)=\sigma_i(H)$. Note that, for every hyperplane $H$, there exist distinct $i,j \in \{1,2,3\}$ such that $\sigma_i(H)=\sigma_j(H)$. Thus, the element $\sigma'$ is entirely determined by the triple $\{\sigma_1,\sigma_2,\sigma_3\}$ and $\sigma'=\sigma$. This proves uniqueness, and finally the fact that $Z'$ is a median graph. Claim~\ref{claim:ZMedian} is thus proved.

\medskip \noindent
It remains to prove the second assertion of our proposition. The fact that the convex hull of $X$ in $Z$ coincides with $Z$ is clear since, for every $p \in N(A) \cap N(B)$, the vertex $\gamma(p)$ belongs to the geodesic $[\alpha(p), \gamma(p)] \cup [\gamma(p),\beta(p)]$ connecting $\alpha(p),\beta(p) \in X$. 

\medskip \noindent
Finally, two hyperplanes $J_1,J_2$ in $X$ are transverse in $Z$ if and only if there exist intersecting edges $e_1 \in J_1$ and $e_2 \in J_2$ that span a $4$-cycle. If such a $4$-cycle lies in $X$, then $J_1,J_2$ are transverse in $X$. Otherwise, the only possibility, by construction of $Z$, is that $e_1,e_2$ belong to $A,B$. Thus, if two hyperplanes $J_1,J_2$ are transverse in $Z$, then either they are transverse in $X$ or $\{J_1,J_2\} = \{A,B\}$. The converse is clear, proving the second assertion of our proposition. 
\end{proof}

\noindent
Finally, we show that our swelling satisfies the universal property given by Theorem~\ref{thm:Swelling}. It justifies that our definition of swelling is the most natural one. 

\begin{prop}\label{prop:ExtensionTwo}
Let $X$ be a median graph and $A,B$ two tangent hyperplanes. Let $Z$ denote the swelling of $X$ relative to $\{A,B\}$. For every median graph $Y$, every parallel-preserving map $\psi : X \to Y$ satisfying $\psi(A) \pitchfork \psi(B)$ admits a unique parallel-preserving extension $Z \to Y$.
\end{prop}

\begin{proof}
Let $\alpha$ (resp. $\beta$) denote the canonical involution of $N(A)$ (resp. $N(B)$). For every $p \in N(A) \cap N(B)$, let $\gamma(p)$ denote the fourth vertex of the $4$-cycle spanned by $p,\alpha(p),\beta(p)$. Notice that, given a $p \in N(A) \cap N(B)$, the fact that $\psi(A)$ and $\psi(B)$ are transverse implies that $\psi(p),\psi(\alpha(p)),\psi(\beta(p))$ span a $4$-cycle in $Y$. Assuming that there exists a parallel-preserving map $\xi : Z \to Y$ extending $\psi$, necessarily $\xi$ has to send $\gamma(p)$ to the fourth vertex of this $4$-cycle. Thus, there exists at most one parallel-preserving extension of $\psi$. Defining $\xi : Z \to Y$ like this, we need to verify that $\xi$ is parallel-preserving. Clearly, $\xi$ sends vertices to vertices. Notice that it also sends edges to edges. Indeed:
\begin{itemize}
	\item Because $\psi$ is parallel-preserving, $\xi$ sends an edge of $X$ to an edge of $Y$.
	\item Given a $p \in N(A) \cap N(B)$, consider the edge $e:=[\alpha(p),\gamma(p)]$ or $[\beta(p),\gamma(p)]$. By construction, $\xi$ sends the $4$-cycle $p,\alpha(p),\beta(p),\gamma(p)$ to a $4$-cycle. A fortiori, it sends $e$ to an edge.
	\item Given two adjacent vertices $p,q \in N(A) \cap N(B)$, consider the edge $e:=[\gamma(p),\gamma(q)]$. By construction, $\xi(p)$ (resp.\ $\xi(q)$) is separated from $\xi(\gamma(p))$ (resp.\ $\xi(\gamma(q))$) exactly by $A$ and $B$. Consequently, the hyperplanes separating $\xi(\gamma(p))$ and $\xi(\gamma(q))$ coincide with the hyperplanes separating $\xi(p)=\psi(p)$ and $\xi(q)=\psi(q)$. Since $\psi$ sends $[p,q]$ to an edge, we conclude that $\xi(\gamma(p))$ and $\xi(\gamma(q))$ are adjacent in $Y$.
\end{itemize}
Now, we claim that, for every edge $e:=[a,b] \subset Z$, there exists an edge $e' \subset X$ parallel to $e$ such that $\xi(e)$ and $\xi(e')$ are parallel. 
\begin{itemize}
	\item If $e$ already lies in $X$, it suffices to take $e':=e$.
	\item Assume that $e$ has only one vertex not in $X$, say $a \in X$ but $b \notin X$. In other words, there exists some $p \in N(A) \cap N(B)$ such that $a=\alpha(p)$ or $\beta(p)$ and $b=\gamma(p)$. Up to switching $A$ and $B$, say that $a=\alpha(p)$. By construction, $\xi$ sends the $4$-cycle spanned by $p,\alpha(p),\beta(p),\gamma(p)$ to a $4$-cycle, so the edge $e':=[p,\beta(p)]$, which lies in $X$ and is parallel to $e$, is sent by $\xi$ to an edge parallel to $\xi(e)$.
	\item Assume that $a,b \notin X$. In other words, there exist two adjacent vertices $p,q \in N(A) \cap N(B)$ such that $a=\gamma(p)$ and $b=\gamma(q)$. By construction, $\xi$ sends the $4$-cycle spanned by $p,\alpha(p),\beta(p),\gamma(p)$ (resp. $q, \alpha(q),\beta(q),\gamma(q)$) to a $4$-cycle. Moreover, because $\xi$ sends edges to edges, we know that $\xi(p)$ (resp. $\xi(\alpha(p))$, $\xi(\beta(p))$, $\xi(\gamma(p))$) is adjacent to $\xi(q)$ (resp. $\xi(\alpha(q))$, $\xi(\beta(q))$, $\xi(\gamma(q))$). Consequently, the edges $[\xi(\gamma(p)),\xi(\gamma(q))]$ and $[\xi(p),\xi(q)]$ are adjacent. Thus, it suffices to take $e':=[p,q]$. 
\end{itemize}
Our claim is sufficient to conclude that $\xi$ is parallel-preserving. Indeed, given two parallel edges $e_1,e_2 \subset Z$, our claim yields two edges $e_1',e_2' \subset X$ parallel respectively to $e_1,e_2$ such that $\xi(e_1),\xi(e_2)$ are parallel respectively to $\xi(e_1'),\xi(e_2')$. Since $e_1'$ and $e_2'$ are necessarily parallel, and because $\xi(e_1')= \psi(e_1')$ and $\xi(e_2')=\psi(e_2')$ must be parallel as $\psi$ is parallel-preserving, we conclude that $\xi(e_1)$ and $\xi(e_2)$ are parallel in $Y$, as desired.
\end{proof}

\noindent
It is worth noticing that, as justified by the next example, in general there is no canonical way to swell two hyperplanes that are not tangent.

\begin{remark}\label{remark:NotCanonicalTangent}
Let $X$ be a path of length four, and let $A,B,C,D$ denote its successive hyperplanes. In order to swell $A$ and $D$, we have two natural constructions. See Figure~\ref{NoSwell}. First, we can fold $B$ and $C$, in order to make $A$ and $D$ tangent, and then swell $A$ and $D$. The median graph $Y$ thus obtained is a edge and a $4$-cycle connected by a common vertex. Next, we can make $A$ transverse to both $B$ and $C$, in order to make $A$ and $D$ tangent, and then swell $A$ and $D$. The median graph $Z$ thus obtained is a chain of three $4$-cycles. Unfortunately, there is no median graph $W$ in which $A,D$ coincide so that $X \to Y$ and $X \to Z$ factor through $Z$ via parallel-preserving maps. 

\begin{figure}[h!]
\begin{center}
\includegraphics[width=0.6\linewidth]{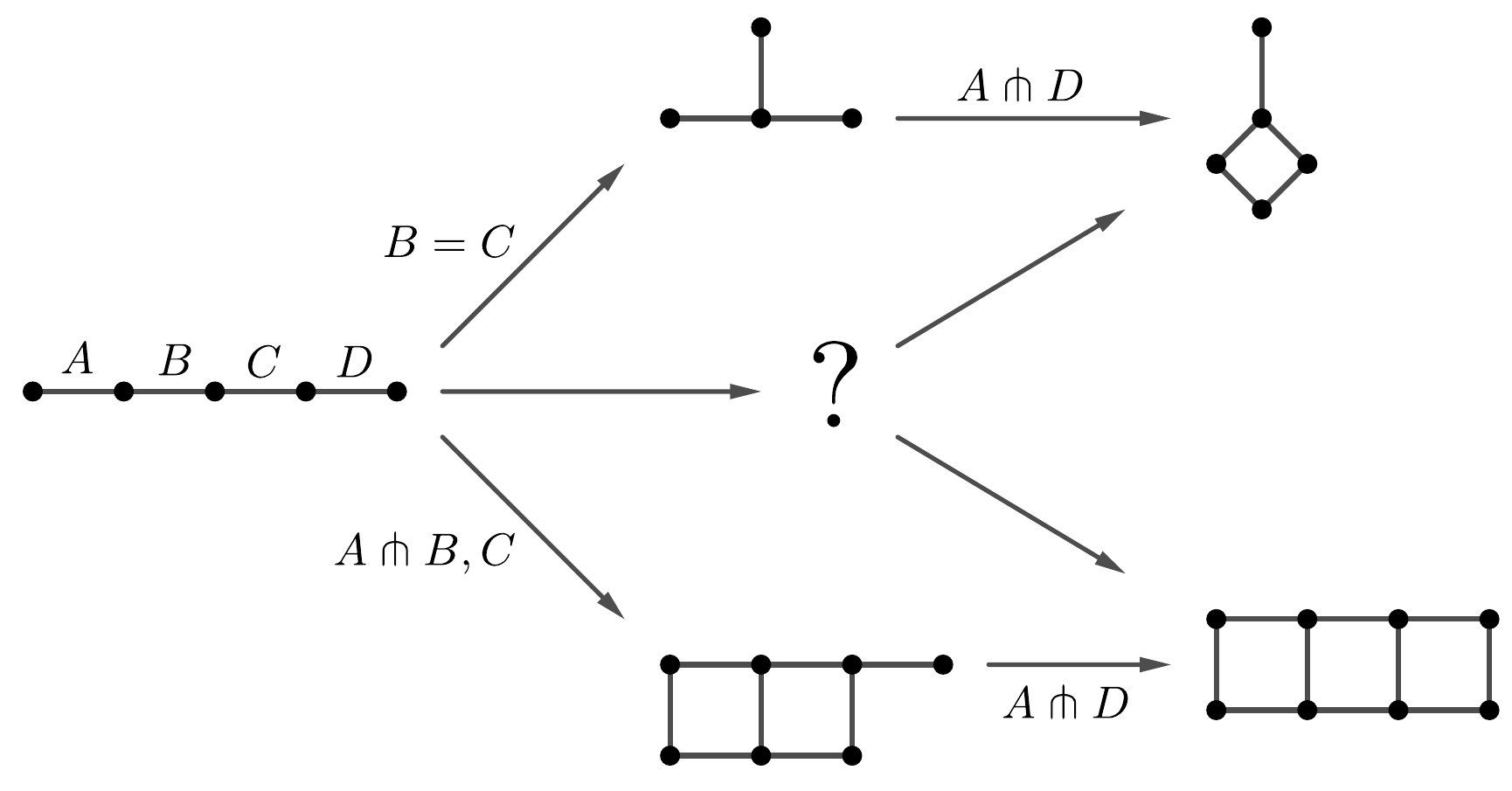}
\caption{A configuration where there is no canonical way to swell two hyperplanes.}
\label{NoSwell}
\end{center}
\end{figure}
\end{remark}

\subsection{Swelling hyperplanes}\label{section:Sw}

\noindent
So far, we have seen how to make two tangent hyperplanes transverse. Now, given an arbitrary collection of pairs of tangent hyperplanes, we prove Theorem~\ref{thm:Swelling} by swelling the pairs successively. A priori, the median graph thus constructed depends on the enumeration we use to swell the pairs of hyperplanes. However, it will follow from the universal property given by Proposition~\ref{prop:ExtensionTwo} that one obtains the same median graph independently of the enumeration. 

\begin{proof}[Proof of Theorem~\ref{thm:Swelling}.]
Up to adding redundancy in $\mathcal{P}$ and saying that swelling a pair of transverse hyperplanes does not modify the median graph, we can assume that $\mathcal{P}$ is infinite. Fix an enumeration $\mathcal{P}= \{ \{A_i, B_i\}, i \geq 1\}$ and define by induction a sequence of graphs $(Z_i)_{i \geq 0}$:
\begin{itemize}
	\item $Z_0$ is just $X$;
	\item if $Z_i$ is already defined, $Z_{i+1}$ is the swelling of $Z_i$ relative to $\{A_{i+1},B_{i+1}\}$.
\end{itemize}
It follows from Proposition~\ref{prop:ExtensionTwo} that, given $i<j$, the isometric embedding $X \hookrightarrow Z_j$ extends to a parallel-preserving map $Z_i \hookrightarrow Z_j$ since $A_i,B_i$ are transverse in $X_j$; moreover, this extension turns out to be an isometric embedding since it sends distinct hyperplanes to distinct hyperplanes. Thus, it makes sense to define $Z$ as the direct limit of the $Z_i$. It is clearly a median graph.

\medskip \noindent
Now, let $Y$ be a median graph and $\psi : X \to Y$ a parallel-preserving map such that $\psi(A_i),\psi(B_i)$ are transverse for every $i \geq 1$. It follows from Proposition~\ref{prop:ExtensionTwo} that, for every $i \geq 1$, $\psi$ admits a unique parallel-preserving extension $\xi_i : Z_i \to Y$. This uniqueness implies that, given $i<j$, the restriction of $\xi_j$ to $Z_i$, thought of as a subgraph of $Z_j$, agrees with $\xi_i$. Thus, the $\xi_i$ define a unique parallel-preseving extension $\xi : Z \to X$ of~$\psi$.

\medskip \noindent
Thus, we have prove the first assertion of our theorem. In order to prove the second assertion, first notice that, by construction of $Z$ and as a consequence of the second assertion of Proposition~\ref{prop:SwellingTwo}, the convex hull of $X$ in $Z$ is $Z$ entirely and two hyperplanes $A,B$ of $Z$ are transverse in $Z$ if and only if they are transverse in $X$ or $\{A,B\} \in \mathcal{P}$. Let $Z'$ denote another median graph containing $X$ isometrically and satisfying these two properties. The inclusion map $X \hookrightarrow Z'$ extends to a parallel-preserving map $\xi : Z \to Z'$. Because $\xi$ sends isometrically the copy of $X$ in $Z$ onto the copy of $X$ in $Z'$, and since the convex hulls of $X$ in $Z$ and $Z'$ are respectively $Z$ and $Z'$, it follows that $\xi$ induces a bijection from the hyperplanes of $Z$ to the hyperplanes of $Z'$. Moreover, as a consequence of the properties satisfied by $Z$ and $Z'$, this bijection preserves transversality and non-transversality. We conclude from Lemma~\ref{lem:WhenIso} that $\xi$ yields an isometry $Z \to Z'$. 
\end{proof}

\noindent
Since Theorem~\ref{thm:Swelling} characterises the median graph $Z$ uniquely up to isometry, the following definition makes sense:

\begin{definition}
Let $X$ be a countable median graph and $\mathcal{P}$ a collection of pairs of tangent hyperplanes in contact. The median graph given by Theorem~\ref{thm:Swelling} is the \emph{swelling of $X$ relative to $\mathcal{P}$}. 
\end{definition}

\noindent
A direct description will be given in Section~\ref{section:SwellingExplicit}.

\subsection{Equivariance}\label{section:SwellingEqui}

\noindent
Typically, there will be groups acting on our median graphs, and we would like our swellings to be equivariant under the actions. Our next proposition shows that the universal property satisfied by swellings automatically implies that our constructions are equivariant under group actions. 

\begin{prop}\label{prop:equivarianceswelling}
Let $G$ be a group acting on a median graph $X$ and $\mathcal{P}$ a collection of pairs of tangent hyperplanes. Let $Z$ denote the swelling of $X$ relative to $\mathcal{P}$. 
\begin{itemize}
	\item Every isometry $X \to X$ preserving $\mathcal{P}$ extends uniquely to an isometry $Z \to Z$. Thus, the action $G \curvearrowright X$ naturally extends to an action $G \curvearrowright Z$.
	\item Let $H$ be a group acting on a median graph $Y$ and $\psi : X \to Y$ a $\varphi$-equivariant parallel-preserving map for some morphism $\varphi : G \to H$. Then there exists a $\varphi$-equivariant parallel-preserving map $\xi : Z \to Y$ such that $\psi = \xi  \circ \iota$. 
\end{itemize}
\end{prop}

\begin{proof}
The first assertion follows from the universal property satisfied by $Z$, as given by Theorem~\ref{thm:Swelling}, since isometric embeddings are parallel-preserving. For the second assertion, we need to verify that the unique extension $\xi$ of $\psi$ is $\varphi$-equivariant. Given a $g \in G$, consider the map $\varphi(g)^{-1} \xi ( g \cdot) : Z \to Y$. This map is parallel-preserving and it coincides with $\psi$ on $X$ since $\psi$ is $\varphi$-equivariant. The uniqueness of $\xi$ implies that $\xi(g \cdot)= \varphi(g) \xi( \cdot)$. 
\end{proof}

\subsection{An explicit description}\label{section:SwellingExplicit}

\noindent
We saw during the proof of Theorem~\ref{thm:Swelling} that the swelling of a median graph $X$ relative to a collection $\mathcal{P}$ of pairs of tangent hyperplanes can be constructed by enumerating $\mathcal{P}$ and swelling the pairs successively as explained in Section~\ref{section:SwellingTwo}. It is also possible to describe the swelling relative to $\mathcal{P}$ directly. This will be useful in Section~\ref{section:Freemediancocompact}.

\medskip \noindent
Following the definition introduced during the proof of Theorem~\ref{thm:Swelling}, fix a countable median graph $X$ and a collection $\mathcal{P}$ of pairs of tangent hyperplanes, and define a \emph{spot} (\emph{relative to $\mathcal{P}$}) as a map $\sigma$ such that:
\begin{itemize}
	\item for every hyperplane $J$ of $X$, $\sigma(J)$ is a halfspace delimited by $J$;
	\item $\sigma(A) \cap \sigma(B) \neq \emptyset$ for all hyperplanes $\{A,B\} \notin \mathcal{P}$;
	\item $\sigma$ differs from some (or equivalently, any) principal orientation on only finitely many hyperplanes.
\end{itemize}
Let $M$ denote the graph whose vertices are the spots and whose edges connect two spots whenever they differ on a single hyperplane. The copy of $X$ in $M$ is given by the principal orientations. 

\begin{prop}\label{prop:SwellingDirect}
Let $Z$ denote the swelling of $X$ relative to $\mathcal{P}$. The embedding $X \hookrightarrow M$ extends to an isometry $Z \to M$. 
\end{prop}

\noindent
The case $\mathcal{P}=\{A,B\}$ is contained in the proof of Proposition~\ref{prop:SwellingTwo}. The general case is a straightforward adaptation of the argument. In a few words, Claims~\ref{claim:DistanceSpot} and~\ref{claim:ZMedian} can be reproduced word for word and yield:

\begin{claim}
The distance between two spots in $M$ coincides with the number of hyperplanes on which they differ.
\end{claim}

\begin{claim}
The graph $M$ is median. 
\end{claim}

\noindent
Notice that, if $\{A,B\} \in \mathcal{P}$, then $A$ and $B$ are transverse in $M$. Indeed, given a vertex $x \in N(A) \cap N(B)$, let $a \in N(A)$ and $b \in N(B)$ denote its neighbours separated from it by $A$ and $B$. Then the map obtained from (the principal orientation given by) $x$ by modifying its values at $A$ and $B$ defines a spot, which induces with (the principal orientations given by) $x$, $a$, and $b$ a $4$-cycle in $M$. Therefore, $A$ and $B$ are transverse in $M$. (Alternatively, one can say that $M$ contains the graph of spots relative to $\{A,B\}$, which coincides with swelling of $X$ relative to $\{A,B\}$ according to the proof of Proposition~\ref{prop:SwellingTwo}.) As a consequence, it follows from Theorem~\ref{thm:Swelling} that the embedding $X \to M$ extends to a parallel-preserving map $Z \to M$. In order to show that this map is an isometry, we need to verify that the convex hull of $X$ in $M$ coincides with $M$ and that two hyperplanes $A,B$ of $X$ are transverse in $M$ if and only if they are transverse in $X$ or $\{A,B\} \in \mathcal{P}$. The first assertion amounts to saying that every hyperplane of $M$ crosses $X$. This follows from Claim~\ref{claim:HyperplaneSpot} below. The second assertion is also an easy consequence of this claim, and is recorded by Claim~\ref{claim:TransverseSpot}. 

\medskip \noindent
Thus, the proofs of the next two claims conclude the proof of Proposition~\ref{prop:SwellingDirect}.

\begin{claim}\label{claim:HyperplaneSpot}
Let $[x,x']$ and $[y,y']$ be two edges from $M$. Let $A$ (resp.\ $B$) denote the hyperplane of $X$ on which $x$ and $x'$ (resp.\ $y$ and $y'$) differ. The edges $[x,x']$ and $[y,y']$ are parallel in $M$ if and only if $A=B$.
\end{claim}

\begin{proof}
First, assume that $[x,x']$ and $[y,y']$ are parallel. It suffices to consider the case where our two edges are opposite sides in a $4$-cycle, the general case following by induction. Up to switching $x$ and $x'$, assume that $x$ is adjacent to $y$. Let $P$ (resp.\ $Q$) denote the hyperplane on which $x$ and $y$ (resp.\ $x'$ and $y'$) differ. Because $y'$ differs from $y$ only on $B$ and that $y$ differs from $x$ only on $P$, it follows that $y'$ differs from $x$ only on $\{B,P\}$. Because $y'$ differs from $x'$ only on $Q$ and that $x'$ differs from $x$ only on $A$, it follows that $y'$ differs from $x$ only on $\{A,Q\}$. Hence $\{A,Q\}=\{B,P\}$. Notice that we cannot have $P=A$, since otherwise we would have $y=x'$. Therefore, $P=Q$ and $A=B$. 

\medskip \noindent
Conversely, assume that $A=B$. Up to switching $x$ and $x'$, we assume that $x(A)=y(A)$. We argue by induction on the number of hyperplanes on which $x$ and $y$ differ. Fix a hyperplane $B$ on which $x$ and $y$ differ such that $x(B)$ is minimal with respect to the inclusion in $\{x(J) \mid J \text{ hyperplane on which differ $x$ and $y$}\}$. Let $p$ denote the map obtained from $x$ by modifying its value at $B$ and let $q$ denote the map obtained from $p$ by modifying its value at $A$. If $p$ and $q$ are spots, then $[p,q]$ yields an edge opposite to $[x,x']$ in a $4$-cycle where $p$ differs from $y$ on less hyperplanes. This will conclude the proof.

\medskip \noindent
Let us verify that $p$ is a spot. Otherwise, there exist two hyperplanes $J$ and $H$ such that $\{J,H\} \notin \mathcal{P}$ and $p(J) \cap p(H)=\emptyset$. Because $p$ differs from the spot $x$ only at $B$, necessarily $J$ or $H$ is $B$, say $H=B$ (and $J \neq B$). Then
$$x(J) \cap x(B)^c = p(J) \cap p(B) = \emptyset, \text{ hence } x(J) \subset x(B),$$
which contradicts our choice of $B$. So $p$ is indeed a spot. Next, if $q$ is not a spot, then there exist two hyperplanes $J$ and $H$ such that $\{J,H\}\notin \mathcal{P}$ and $q(J) \cap q(H)= \emptyset$. Because $q$ differs from the spot $p$ only at $A$, necessarily $J$ or $H$ is $A$, say $H=A$ (and $J \neq A$). Notice that $q(A)=p(A)^c = x(A)^c$ and $q(J)=x(J)$ if $J \neq B$ and $x(B)^c$ otherwise. In the latter case, where $J=B$, we have
$$\emptyset = x(B)^c \cap x(A)^c = y(B) \cap y(A), \text{ hence } \{A,B\} \in \mathcal{P}$$
because $y$ is a spot, but this is impossible since $\{A,B\}= \{J,H\} \notin \mathcal{P}$. In the former case, where $J \neq B$, we have
$$\emptyset = x(J) \cap x(A)^c = x'(J) \cap x'(A), \text{ hence } \{J,A\}\in \mathcal{P}$$
because $x'$ is a spot, which is again a contradiction since $\{J,A\}= \{J,H\} \notin \mathcal{P}$. Thus, $q$ is also a spot. 
\end{proof}

\begin{claim}\label{claim:TransverseSpot}
Let $A$ and $B$ be two hyperplanes of $X$. If $A$ and $B$ are transverse in $M$, then they are transverse in $X$ or $\{A,B\} \in \mathcal{P}$. 
\end{claim}

\begin{proof}
Let $(r,s,t,u)$ be a $4$-cycle in $M$ crossed by $A$ and $B$. To fix the notation, say that $[r,u]$ (and $[s,t]$) is crossed by $A$ and that $[r,s]$ (and $[t,u]$) is crossed by $B$. It follows from Claim~\ref{claim:HyperplaneSpot} that $r$ and $u$, as well as $s$ and $t$, only differ on $A$; and that $r$ and $s$, as well as $t$ and $u$, only differ on $B$. Consider the four intersections
$$\left\{ \begin{array}{l} r(A) \cap r(B) \\ u(A) \cap u(B)=r(A)^c \cap r(B) \\ s(A) \cap s(B) = r(A) \cap r(B)^c \\ t(A) \cap t(B) = r(A)^c \cap r(B)^c \end{array} \right..$$
Either the four intersections are non-empty, and $A,B$ must be transverse in $X$, or at least one intersection is empty, and the definition of spots imposes that $\{A,B\} \in \mathcal{P}$. 
\end{proof}

\section{Factorisations through isometric embeddings}\label{section:isoembedding}

In this section, we prove our main result, namely Theorem~\ref{Intro:mainthm} (see~Theorem~\ref{thm:BigFactor}). Theorem~\ref{Intro:mainthm} constitutes our analogue of Stallings' folds for trees.

\subsection{With infinitely many foldings and swellings}

\noindent
In our next statement, we claim that every parallel-preserving map between two median graphs factors canonically through a sequence of (possibly infinitely many) foldings and swellings as an isometric embedding. 

\begin{thm}\label{thm:BigFactor}
Let $G$ be a group acting on two countable median graphs $X,Y,$ and $\psi : X \to Y$ a $G$-equivariant parallel-preserving map. There exist a (possibly infinite) sequence of $G$-equivariant foldings and swellings $\eta : X \to \cdots \to Z$ and a $G$-equivariant isometric embedding $\iota : Z \to Y$ such that $\psi = \iota \circ \eta$. Moreover, $\iota(Z)$ coincides with the median hull of $\psi(X)$ in $Y$.
\end{thm}

\noindent
In the theorem, an infinite sequence of foldings and swellings $\eta : X \to \cdots \to Z$ refers to the canonical map from $X$ to the direct limit $Z$ associated to a sequence $X \to X_1, X_1 \to X_2, \ldots$ of foldings and swellings, as given by Definition~\ref{def:DirectLimit}.

\medskip \noindent
The idea is that, if our parallel-preserving map $X \to Y$ is not an isometric embedding, then we know from Lemma~\ref{lem:WhenIso} that there are two hyperplanes in $X$ with the same image in $Y$. The goal is to fold two such hyperplanes and to repeat the process until getting an isometric embedding. However, our two hyperplanes may not be in contact. However, one can make them in contact by adding some transversality (in a non-canonical way though, as justified by Remark~\ref{remark:NotCanonicalFolding}). This is what does the next lemma, which requires the following definition.

\begin{definition}
Let $G$ be a group acting on a median graph $X$ and $\mathcal{P}$ a collection of pairs of hyperplanes in contact (resp.\ transverse). The folding (resp.\ swelling) relative to $\mathcal{P}$ is \emph{$G$-atomic} if if $\mathcal{P}$ consists of a single $G$-orbit of pairs of hyperplanes, i.e.\ $\mathcal{P}= \{ g \{A,B\}, g \in G\}$ for some hyperplanes $A,B$ of $X$.
\end{definition}

\noindent
In other words, a $G$-atomic folding (resp.\ swelling) is a ``smallest'' possible folding (resp.\ swelling) that is $G$-equivariant. Notice that, by decomposing collections of pairs of hyperplanes into $G$-orbits, every folding (resp.\ swelling) can be decomposed as a product of (possibly infinitely many) $G$-atomic foldings (resp.\ swellings). Therefore, $G$-atomic foldings and swellings will not play a role in the proof of Theorem~\ref{thm:BigFactor}. However, they will be central in the next section, where we will be concerned in using as few foldings and swellings as possible. 

\begin{lemma}\label{lem:folduniquepair}
Let $G$ be a group acting on two countable median graphs $X,Y$ and $\psi : X \to Y$ a $G$-equivariant parallel-preserving map. Let $A,B$ be two distinct hyperplanes such that $\psi(A)=\psi(B)$. There exists a finite sequence of $G$-atomic foldings and swellings $\eta \colon X \to \cdots \to Z$ such that $\eta(A)=\eta(B)$ and a unique $G$-equivariant parallel-preserving map $\iota : Z \to Y$ such that $\psi = \iota \circ \eta$.
\end{lemma}

\begin{proof}
We argue by induction over the distance between $A$ and $B$.

\medskip \noindent
First, assume that $d(A,B)=0$. In other words, $A$ and $B$ are in contact. Set $\mathcal{P}=\{ g\{A,B\}, g \in G\}$ and let $Z$ denote the folding of $X$ with respect to $\mathcal{P}$, as given by Theorem~\ref{thm:Foldings}. By the universal property satisfied by $Z$, there exist a $G$-equivariant parallel-preserving map $\eta \colon X \to Z$ satisfying $\eta(A)=\eta(B)$ and a unique $G$-equivariant parallel-preserving map $\iota \colon Z \to Y$ satisfying $\psi = \iota \circ \eta$. This yields the desired conclusion.

\medskip \noindent
From now on, assume that $d(A,B) \geq 1$. Let $H_1,\ldots,H_n$ denote the hyperplanes separating $A$ and $B$ in $X$. For convenience, set $H_0:=A$ and $H_{n+1}:=B$. We distinguish two cases, depending on whether or not two $H_i$ have the same image under $\psi$.

\medskip \noindent
First, assume that there exist distinct $i,j \in \{0,\ldots,n\}$ such that $\psi(H_i)=\psi(H_j)$. Because $d(H_i,H_j) < d(A,B)$, we know by induction that there exists a finite sequence of $G$-atomic foldings and swellings $\eta_0 \colon X \to \ldots \to Z_0$ satisfying $\eta_0(H_i)=\eta_0(H_j)$ and a unique $G$-equivariant parallel-preserving map $\iota_0 : Z_0 \to Y$ satisfying $\psi = \iota_0 \circ \eta_0$. Notice that $d(\eta_0(A),\eta_0(B)) < d(A,B)$. Indeed, consider a geodesic $\gamma$ in $X$ connecting two vertices of $N(A)$ and $N(B)$ at minimal distance. Then the hyperplanes crossing $\gamma$ are exactly the hyperplanes separating $A$ and $B$, namely the $H_i$. But $\eta_0(\gamma)$, which connects $N(\eta_0(A))$ and $N(\eta_0(B))$, is not a geodesic since it crosses one hyperplane twice, namely $\eta_0(H_i)=\eta_0(H_j)$. This implies the desired inequality. Thus, we know by induction that there exist a finite sequence of $G$-atomic foldings and swellings $\eta_1 \colon Z_0 \to Z$ such that $\eta_1(\eta_0(A))=\eta_1(\eta_0(B))$ and a unique $G$-equivariant parallel-preserving map $\iota \colon Z \to Y$ satisfying $\iota_0=\iota \circ \eta_1$.

\hspace{3cm} \xymatrix{
X \ar[ddrr]_\eta \ar[drr]^{\eta_0} \ar[rrrr]^\psi &&  && Y \\
&& Z_0 \ar[urr]^{\iota_0} \ar[d]^{\eta_1} && \\
&& Z \ar[uurr]_\iota && 
}

\noindent
Set $\eta=\eta_1 \circ \eta_0$. Then $\eta$ is a finite sequence of $G$-atomic foldings and swellings satisfying $\eta(A)=\eta(B)$ and $\iota$ is the unique $G$-equivariant parallel-preserving map $Z \to Y$ satisfying $\psi=\iota \circ \eta$. This concludes the proof of our lemma in this case.

\medskip \noindent
Next, assume that $\psi(H_i)\neq \psi(H_j)$ for all distinct $i,j \in \{0,\ldots,n\}$. Consider a geodesic $\gamma$ in $X$ connecting two vertices of $N(A)$ and $N(B)$ at minimal distance. The hyperplanes crossing $\gamma$ are exactly the hyperplanes separating $A$ and $B$, namely the $H_i$. Consequently, the hyperplanes crossed by $\psi(\gamma)$ are the $\psi(H_i)$. Since these hyperplanes are pairwise distinct, it follows that $\psi(\gamma)$ is a geodesic in $Y$ connecting two vertices in the carrier of $\psi(A)=\psi(B)$. Because carriers are convex, we deduce that the $\psi(H_i)$ must be transverse to $\psi(A)=\psi(B)$. Up to reindexing the $H_i$, assume that $H_1$ is the hyperplane containing the edge of $\gamma$ intersecting $N(A)$, which implies that $H_1$ is tangent to $A$. Then we know that $\psi(H_1)$ is transverse to $\psi(A)$ (but $H_1$ is not transverse to $A$ since it separates $A$ and $B$). Therefore, we can set $\mathcal{P}:=\{ g\{A,H_1\} , g \in G \}$ and let $Z_0$ be the swelling of $X$ with respect to $\mathcal{P}$, as given by Theorem~\ref{thm:Swelling}. By the universal property satisfied by $Z_0$, there exist a $G$-equivariant parallel-preserving map $\eta_0 \colon X \to X_0$ satisfying $\eta_0(A)=\eta_0(B)$ and a unique $G$-equivariant parallel-preserving map $\iota_0 \colon X_0 \to Y$ satisfying $\psi = \iota_0 \circ \eta_0$.

\medskip \noindent
Notice that $d(\eta_0(A),\eta_0(B))< d(A,B)$. Indeed, the path $\eta_0(\gamma)$ connects the carriers of $\eta_0(A)$ and $\eta_0(B)$, but has an edge in the carrier of $\eta_0(A)$ (namely, the edge contained in $\eta_0(H_1)$), so it cannot be a shortest path connecting the carriers of $\eta_0(A)$ and $\eta_0(B)$. Therefore, we know by induction that there exist a finite sequence of $G$-atomic foldings and swellings $\eta_1\colon X_0 \to Z$ such that $\eta_1(\eta_0(A))=\eta_1(\eta_0(B))$ and a $G$-equivariant map $\iota_1 \colon Z \to Y$ such that $\iota_0=\iota_1 \circ \eta_1$. 

\hspace{3cm} \xymatrix{
X \ar[ddrr]_\eta \ar[drr]^{\eta_0} \ar[rrrr]^\psi &&  && Y \\
&& Z_0 \ar[urr]^{\iota_0} \ar[d]^{\eta_1} && \\
&& Z \ar[uurr]_\iota && 
}

\noindent
Set $\eta=\eta_1 \circ \eta_0$. Then $\eta$ is a finite sequence of $G$-atomic foldings and swellings satisfying $\eta(A)=\eta(B)$ and $\iota$ is the unique $G$-equivariant parallel-preserving map $Z \to Y$ satisfying $\psi=\iota \circ \eta$. This concludes the proof of our lemma in this case.
\end{proof}

\begin{proof}[Proof of Theorem~\ref{thm:BigFactor}.]
Let $Y_0$ denote the median hull of $\psi(X)$ in $Y$. Our goal is to apply Lemma~\ref{lem:folduniquepair} iteratively in order to fold all the pairs of hyperplanes of $X$ that have the same image in $Y_0$ (or equivalently, in $Y$) under $\psi$. For this purpose, set 
$$\mathcal{Q}:= \{ \{A,B\} \mid \text{ distinct hyperplanes $A,B$ such that } \psi(A)=\psi(B)\}$$
and fix an enumeration $\mathcal{Q}= \{ \{A_i,B_i\}, i \geq 1\}$.

\medskip \noindent
First, we apply Lemma~\ref{lem:folduniquepair} in order to get a median graph $Z_1$, a sequence of $G$-equivariant foldings and swellings $\eta_1 : X \to Z_1$ satisfying $\eta_1(A_1)=\eta_1(B_1)$, and a $G$-equivariant parallel-preserving map $\iota_1 : Z_1 \to Y_0$. Next, let $n_1$ denote the smallest integer satisfying $\eta_1(A_{n_1}) \neq \eta_1(B_{n_1})$. Because we must have $\iota_1(\eta_1(A_{n_1}))=\iota_1(\eta_1(B_{n_1}))$, we can apply Lemma~\ref{lem:folduniquepair} again (to $\eta_1(A_{n_1})$ and $\eta_1(B_{n_1})$) and obtain a median graph $Z_2$, a sequence of $G$-equivariant foldings and swellings $\eta_2 : Z_1 \to Z_2$ satisfying $\eta_2(\eta_1(A_{n_1}))=\eta_2(\eta_1(B_{n_1}))$, and a $G$-equivariant parallel-preserving map $\iota_2 : Z_2 \to Y_0$. Next, we iterate the process.

\hspace{3cm}\xymatrix{
X \ar[rrdddd]_\eta \ar[rrrr]^\psi \ar[drr]^{\eta_1} && && Y_0 \ar[r] & Y \\
&& Z_1 \ar[rru]^{\iota_1} \ar[d]^{\eta_2} && \\
&& Z_2 \ar[rruu]_{\iota_2} \ar[d]^{\eta_3} && \\
&& \vdots \ar[d] && \\
&& Z \ar[rruuuu]_\iota &&
}

\noindent
Let $Z$ denote the direct limit of $(X,\eta_1), (Z_1, \eta_2), \ldots$ as given by Definition~\ref{def:DirectLimit}. The $\eta_i$ extends to a $\eta : X \to Z$, which is by construction a sequence of $G$-equivariant foldings and swellings, and the $\iota_i$ extends to a $G$-equivariant parallel-preserving map $\iota : Z \to Y_0$. 

\medskip \noindent
Let us justify that $\iota$ is an isometric embedding. According to Lemma~\ref{lem:WhenIso}, this amounts to showing that $\iota$ is injective on the set of the hyperplanes of $Z$. If it is not the case, then one can find two distinct hyperplanes $A,B$ of $Z$ such that $\iota(A)=\iota(B)$. Because $\eta$ is a sequence of foldings and swellings, there must exist two hyperplanes $P,Q$ of $X$ such that $\eta(P)=A$ and $\eta(Q)=B$. Necessarily, $P$ and $Q$ are distinct, and moreover 
$$\psi(P)= \iota ( \eta(P)) = \iota(A)= \iota(B)= \iota( \eta(Q)) = \psi(Q),$$
so there must exist some $i \geq 1$ such that $\{P,Q\}= \{A_i,B_i\}$. By construction, there must exist some $j \geq 1$ such that $A_i$ and $B_i$ have the same image in $Z_j$ under $\eta_j \circ \cdots \circ \eta_1$. But this implies that $\eta(P)$ and $\eta(Q)$ coincide, i.e.\ $A=B$. Thus, $\iota$ is indeed an isometric embedding.

\medskip \noindent
Finally, notice that $\iota(Z)$ is a subgraph of $Y_0$ that contains $\psi(X)$ and that is median. By definition of the median hull, the equality $\iota(Z)=Y_0$ must hold.
\end{proof}

\noindent
It is worth noticing that Theorem~\ref{thm:BigFactor} has an analogue for isometric embeddings with convex images. Its proof follows the same lines.

\begin{thm}\label{thm:BigFactorTwo}
Let $G$ be a group acting on two countable median graphs $X,Y$ and $\psi : X \to Y$ a $G$-equivariant parallel-preserving map. There exist a (possibly infinite) sequence of $G$-equivariant foldings and swellings $\eta : X \to \cdots \to Z$ and a $G$-equivariant isometric embedding $\iota : Z \to Y$ with convex image such that $\psi = \iota \circ \eta$. Moreover, $\iota(Z)$ coincides with the convex hull of $\psi(X)$ in $Y$. 
\end{thm}

\noindent
Roughly speaking, one can apply Theorem~\ref{thm:BigFactor} in order to factor our parallel-preserving map through an isometric embedding, and next apply swellings in order to make the image convex (since it is, a priori, only median). As justified by our next observation, contrary to Theorem~\ref{thm:BigFactor}, it will not be necessary to fold or swell pairs of hyperplanes not in contact.

\begin{lemma}\label{lem:NotInjConv}
Let $\psi : X \to Y$ be a parallel-preserving map between two median graphs. If $\psi$ is not an isometric embedding with convex image, then there exist two hyperplanes in contact $A,B$ in $X$ such that either $A,B$ are distinct but $\psi(A)=\psi(B)$ or $A,B$ are not transverse but $\psi(A) \pitchfork \psi(B)$.
\end{lemma}

\begin{proof}
First, assume that $\psi$ is not an isometric embedding. According to Lemma~\ref{lem:WhenIso}, there exist two distinct hyperplanes $A$ and $B$ such that $\psi(A)=\psi(B)$. We choose $A$ and $B$ at minimal distance. If $A$ and $B$ are in contact, then we are done. Otherwise, let $H_1, \ldots, H_n$ denote the hyperplanes separating $A$ and $B$. By minimality, $\psi(A),\psi(H_1), \ldots, \psi(H_n)$ are pairwise distinct. Consequently, if we fix a geodesic $\gamma$ between two vertices minimising the distance between $N(A)$ and $N(B)$, because we know that the hyperplanes that $\gamma$ crosses are exactly the $H_i$, then $\psi(\gamma)$ must be a geodesic in $Y$ since the hyperplanes crossed by $\psi(\gamma)$ are exactly the $\psi(H_i)$. Thus, $\psi(\gamma)$ is a geodesic connecting two vertices in the carrier of $\psi(A)$. By convexity of $N(A)$, the $\psi(\gamma)$ must be contained in the carrier of $\psi(A)$, which implies that the $\psi(H_i)$ must all be transverse to $\psi(A)$. If $H_j$ is the hyperplane containing the edge of $\gamma$ with an endpoint in $N(A)$, then we know that $H_j$ and $A$ are tangent but $\psi(H_j)$ and $\psi(A)$ are transverse.

\medskip \noindent
Next, assume that $\psi$ is an isometric embedding but that $\psi(X)$ is not convex. According to \cite{Chepoi}, we know that there exist two intersecting edges $\psi(e_1),\psi(e_2)$ in $\psi(X)$ that span a $4$-cycle in $Y$ but not in $\psi(X)$. Consequently, if $J_1,J_2$ denote the hyperplanes of $X$ containing $e_1,e_2$, then $J_1$ and $J_2$ are tangent but $\psi(J_1)$ and $\psi(J_2)$ are transverse. 
\end{proof}

\begin{proof}[Proof of Theorem~\ref{thm:BigFactorTwo}.]
Let $Y_0$ denote the convex hull of $\psi(X)$ in $Y$. Our goal is to fold and swell hyperplanes iteratively in order to obtain a parallel-preserving map that sends distinct hyperplanes to distinct hyperplanes and (non-)transverse hyperplanes to (non-)transverse hyperplanes. For this purpose, set 
$$\mathcal{Q}:= \left\{ \{A,B\} ~\left| \begin{array}{c} \text{ hyperplanes $A,B$ either in contact such that } \psi(A)=\psi(B) \\ \text{or tangent such that } \psi(A) \pitchfork \psi(B) \end{array} \right. \right\}.$$
and fix an enumeration $\mathcal{Q}= \{ \{A_i,B_i\}, i \geq 1\}$.

\medskip \noindent
First, we fold or swell $X$ relatively to the $G$-orbit of $\{A_1,B_1\}$ in order to get a median graph $Z_1$, a $G$-equivariant folding or swelling $\eta_1 : X \to Z_1$ satisfying $\eta_1(A_1)=\eta_1(B_1)$ or $\eta_1(A_1) \pitchfork \eta_1(B_1)$, and a $G$-equivariant parallel-preserving map $\iota_1 : Z_1 \to Y_0$. Next, let $n_1$ denote the smallest integer such that $\psi(A_{n_1})=\psi(B_{n_1})$ but $\eta_1(A_{n_1}) \neq \eta_1(B_{n_1})$ or $\psi(A_{n_1}) \pitchfork \psi(B_{n_1})$ but $\eta_1(A_{n_1})$ and $\eta_1(B_{n_1})$ not transverse. Notice that $\eta_1(A_{n_1})$ and $\eta_1(B_{n_1})$ are necessarily in contact because $\eta_1$ is $1$-Lipschitz. Then we can fold of swell the $G$-orbit of $\{\eta_1(A_{n_1}), \eta_1(B_{n_1})\}$ and obtain a median graph $Z_2$, a $G$-equivariant folding or swelling $\eta_2 : Z_1 \to Z_2$ satisfying $\eta_2(\eta_1(A_{n_1}))=\eta_2(\eta_1(B_{n_1}))$ or $\eta_2(\eta_1(A_{n_1}))\pitchfork \eta_2(\eta_1(B_{n_1}))$, and a $G$-equivariant parallel-preserving map $\iota_2 : Z_2 \to Y_0$. Next, we iterate the process.

\hspace{3cm}\xymatrix{
X \ar[rrdddd]_\eta \ar[rrrr]^\psi \ar[drr]^{\eta_1} && && Y_0 \ar[r] & Y \\
&& Z_1 \ar[rru]^{\iota_1} \ar[d]^{\eta_2} && \\
&& Z_2 \ar[rruu]_{\iota_2} \ar[d]^{\eta_3} && \\
&& \vdots \ar[d] && \\
&& Z \ar[rruuuu]_\iota &&
}

\noindent
Let $Z$ denote the direct limit of $(X,\eta_1), (Z_1, \eta_2), \ldots$ as given by Definition~\ref{def:DirectLimit}. The $\eta_i$ extends to a $\eta : X \to Z$, which is by construction a sequence of $G$-equivariant foldings and swellings, and the $\iota_i$ extends to a $G$-equivariant parallel-preserving map $\iota : Z \to Y_0$. 

\medskip \noindent
Let us justify that $\iota$ is an isometric embedding with convex image. If it is not the case, it follows from Lemma~\ref{lem:NotInjConv} that one can find two hyperplanes in contact $A,B$ of $Z$ such that either $A,B$ are distinct but $\iota(A)=\iota(B)$ or $A,B$ are not transverse but $\iota(A) \pitchfork \iota(B)$. Because $\eta$ is a sequence of foldings and swellings, there must exist two hyperplanes $P,Q$ of $X$ such that $\eta(P)=A$ and $\eta(Q)=B$. If $A,B$ are distinct but $\iota(A)=\iota(B)$, then $P$ and $Q$ must be distinct distinct and we have 
$$\psi(P)= \iota ( \eta(P)) = \iota(A)= \iota(B)= \iota( \eta(Q)) = \psi(Q).$$
If $A,B$ are not transverse but $\iota(A) \pitchfork \iota(B)$, then $P$ and $Q$ cannot be transverse (as $\eta$ is parallel-preserving) but $\psi(P)$ and $\psi(Q)$ must be transverse (as $\iota$ is parallel-preserving). In other words, there must exist some $i \geq 1$ such that $\{P,Q\}= \{A_i,B_i\}$. By construction, there must exist some $j \geq 1$ such that the images of $A_i$ and $B_i$ in $Z_j$ under $\eta_j \circ \cdots \circ \eta_1$ either coincide or are transverse. But this implies that $\eta(P)$ and $\eta(Q)$ either coincide or are transverse, i.e.\ $A=B$ or $A \pitchfork B$. Thus, $\iota$ is indeed an isometric embedding with convex image.

\medskip \noindent
Finally, notice that $\iota(Z)$ is a convex subgraph of $Y_0$ that contains $\psi(X)$. By definition of the convex hull, the equality $\iota(Z)=Y_0$ must hold.
\end{proof}

\subsection{With only finitely many foldings and swellings}

\noindent
We saw in the previous section that every parallel-preserving map between two median graphs factors as an isometric embedding through a sequence of foldings and swellings. However, such a sequence is usually infinite. In our next statement, we record a particular case of interest where the sequence of foldings and swellings can be chosen finite. Its analogue for trees can be found in \cite{MR1091614}.

\begin{thm}\label{thm:BF}
Let $G$ be a group acting on two countable median graphs $X,Y$ and let $\psi : X \to Y$ be a $G$-equivariant parallel-preserving map. Assume that $X$ has finitely many orbits of hyperplanes and that $Y$ has finitely generated hyperplane-stabilisers. Then there exists a finite sequence of $G$-atomic foldings and swellings $\eta : X \to \cdots \to Z$ such that $\psi = \iota \circ \eta$ for some $G$-equivariant isometric embedding $\iota : Z \hookrightarrow Y$. 
\end{thm} 

\begin{proof}
Let $H_1,\ldots,H_n$ be representatives of the $G$-orbits of hyperplanes in $X$. By applying Lemma~\ref{lem:folduniquepair} several times, there exists a median graph $Z$ equipped with an action of $G$ by isometries and a finite sequence of $G$-atomic foldings and swellings $\eta \colon X \to Z$ such that, for all $i,j \in \{1,\ldots,n\}$ satisfying $\psi(H_i)=\psi(H_j)$, we have $\eta(H_i)=\eta(H_j)$. Moreover, there exists an equivariant map $\iota \colon Z \to Y$ such that $\psi=\iota \circ \eta$. Thus, up to replacing $X$ by $Z$, we may suppose that the map induced by $\psi$ between the sets of $G$-orbits of hyperplanes of $X$ and $Y$ is injective. 

\medskip \noindent
For every $i \in \{1,\ldots,n\}$, let $S_i$ be a finite generating set for the stabiliser of $\psi(H_i)$ in $Y$ and let $S_i^-$ be the subset of $S_i$ consisting of all the elements of $S_i$ that do not belong to the stabiliser of $H_i$ in $X$. We introduce the complexity 
\[c(\psi)=\sum_{i=1}^n  |S_i^-|.\]
We claim that, if $c(\psi) >0$, then a finite sequence of $G$-atomic foldings and swellings reduces $c(\psi)$. Indeed, let $s \in S_i^-$. Then $\psi(sH_i)=\psi(H_i)$ but $sH_i \neq H_i$. By Lemma~\ref{lem:folduniquepair}, there exists a finite sequence of $G$-atomic foldings and swellings $\eta_0 \colon X \to X_0$ such that $\eta_0(sH_i)=\eta_0(H_i)$. Let $\iota_0 \colon X_0 \to Y$ be the equivariant map such that $\psi=\iota_0 \circ \eta_0$. Then the complexity $c(\iota_0)$ is smaller than the complexity $c(\psi)$. An inductive argument then shows that there exist a finite sequence of $G$-atomic foldings and swellings $\eta \colon X \to X$ and an equivariant map $\iota \colon X \to Y$ such that $\psi=\iota \circ \eta$ and $c(\iota)=0$. Thus, $\iota$ is injective on the set of the hyperplanes that belong to a same orbit. Recall that, by Lemma~\ref{lem:CubulationInductively} and Proposition~\ref{prop:SwellingTwo}, the map $\eta$ is surjective on the set of hyperplanes. Thus, as $\psi$ is injective on the set of $G$-orbits of hyperplanes and as $\psi=\iota \circ \eta$, this implies that $\iota$ is injective on the set of hyperplanes. By Lemma~\ref{lem:WhenIso}, the map $\iota$ is an isometric embedding. 
\end{proof}

\section{Median-cocompact subgroups}\label{section:MedianCocompact}

\noindent
In this section, we record a few applications of the techniques developed so far to the recognition of a specific class of subgroups:

\begin{definition}
Let $G$ be a group acting on a median graph $X$. A subgroup $H \leq G$ is \emph{(strongly) median-cocompact} if it acts cocompactly on some connected median-closed subgraph $Y \subset X$.
\end{definition}

\noindent
In a median graph $X$, a \emph{median-closed} subgraph $Y \subset X$ refers to a subgraph stable under the median operation, i.e.\ the median point of any three vertices of $Y$ also belongs to $Y$. Because connected median-closed subgraphs are isometrically embedded, it follows that median-cocompact subgroups in a group acting geometrically on a median graph are undistorted. Median-cocompact subgroups encompass \emph{convex-cocompact} subgroups, namely subgroups acting cocompactly on convex subgraphs (which coincide with the convex-cocompact subgroups from a CAT(0) perspective, as subcomplexes are convex with respect to the CAT(0) metric if and only if their one-skeletons are convex with respect to the graph metric). As mentioned in the introduction, examples of median-cocompact subgroups include cyclic subgroups, some centralisers, some fixators of group automorphisms, and Morse subgroups.

\medskip \noindent
In the definition above, we distinguish the strong median-cocompactness from the (weak) median-cocompactness defined in \cite{EliaOne} (and appearing implicitly in other places, e.g.\ \cite{MR4218342}), which does not require the median-closed subgraph to be connected. In full generality, the two notions differ. For instance, let $\mathbb{Z}^{(\infty)}$ be a free abelian group of countable rank. Fixing a free basis $B=\{b_1,b_2, \ldots\}$, the Cayley graph of $\mathbb{Z}^{(\infty)}$ is a median graph (namely a product of countably many bi-infinite lines), and, with respect to the canonical action of $\mathbb{Z}^{(\infty)}$, the subgroup $\langle b_1,2b_2,3b_3, \ldots \rangle$ is weakly median-cocompact but not strongly median-compact. Nevertheless, it is worth noticing that, for groups acting properly and cocompactly on median graphs, weak and strong median-cocompactness turns out to be equivalent (as justified by \cite[Propositions~4.1 and~4.11]{EliaTwo}).

\subsection{Free subgroups}\label{section:Freemediancocompact}

\noindent
Given a $G$ acting on trees $X,Y$, there often exists a $G$-equivariant map $X \to Y$ sending vertices to vertices and edges to edges. More precisely, it suffices that elliptic subgroups for $G \curvearrowright X$ are also elliptic for $G \curvearrowright Y$. This explains why foldings are ubiquitous in the study of groups acting on trees. When $X$ and $Y$ are more general median graphs, asking for a $G$-equivariant parallel-preserving map $X \to Y$ is unfortunately much more restrictive. In this section, we illustrate one specific case where such a construction is possible, namely when $X$ is a tree (but not necessarily $Y$), which allows us to prove the following application to free subgroups:

\begin{thm}\label{thm:FreeSub}
Let $G$ be a group acting properly and cocompactly on a median graph $X$. A finitely generated free subgroup $F\leq G$ is median-cocompact if and only if every hyperplane skewered by $F$ has a finitely generated $F$-stabiliser. 
\end{thm}

\noindent
Recall that a hyperplane $J$ is \emph{skewered by $H$} if $h J^+ \subsetneq J^+$ for some halfspace $J^+$ delimited by $J$ and for some $h \in H$. 

\medskip \noindent
Given a free group $F$ acting both on a tree $T$ and on a median graph $X$, it is easy to construct an $F$-equivariant parallel-preserving map $T \to X$. This map turns out to be more than parallel-preserving: it also preserves non-adjacency. This property will be preserved when factorising through foldings and swellings, and will be technically convenient, so we give it a name: 

\begin{definition}
A map $\psi : X \to Y$ between two median graphs $X$ and $Y$ is \emph{chiasmatic} if it is parallel-preserving and sends transverse hyperplanes to transverse hyperplanes.
\end{definition}

\noindent
The difficulty in proving Theorem~\ref{thm:FreeSub} is to verify that coboundedness of the action is preserved by the foldings and swellings involved. This is verified by Propositions~\ref{prop:CocompactSwelling} and~\ref{prop:CocompactFolding} below. 

\begin{prop}\label{prop:CocompactSwelling}
Let $G$ be a group acting on two median graphs $X,Y$ of finite cubical dimension. Let $X \to Y$ a $G$-equivariant chiasmatic map that factors as $X \to Z \to Y$ where $X \to Z$ is a $G$-equivariant swelling. If $X$ has bounded degree and if $G$ acts cocompactly on $X$, then $G$ acts cocompactly on $Z$. 
\end{prop}

\noindent
For clarity, we split the proof of the proposition into two lemmas.

\begin{lemma}\label{lem:SwellCobounded}
Let $X$ be a median graph of degree $\leq N$ and $Z$ another median graph obtained by swelling $X$. Every vertex of $Z$ lies at distance $\leq N$ from $X$. 
\end{lemma}

\begin{proof}
We think of the vertices of $Z$ as orientations as described in Section~\ref{section:SwellingExplicit}. Let $z \in Z$ be an arbitrary vertex. Let $\mathcal{Z}$ denote the set of pairs $\{A,B\}$ of hyperplanes of $X$ such that $z(A) \cap z(B)=\emptyset$. Notice that, because $z$ differs from a principal orientations only on finitely many hyperplanes, $\mathcal{Z}$ is necessarily finite.

\begin{claim}\label{claim:PairwiseTangent}
The hyperplanes appearing in $\mathcal{Z}$ are pairwise in contact.
\end{claim}

\noindent
Otherwise, there exist $\{A,A'\}, \{B,B'\} \in \mathcal{Z}$ such that $A$ and $B$ are not in contact. Notice that $z(A) \cap z(B)$ cannot be empty, since otherwise $\{A,B\}$ would belong to $\mathcal{Z}$ and $A,B$ would have to be tangent. If $z(A)$ and $z(B)$ are nested, say $z(B) \subset z(A)$, then $z(A')$, which is disjoint from $z(A)$, must be disjoint from $z(B)$. Hence $\{A',B\} \in \mathcal{Z}$. But this is impossible since it would imply that $A',B$ are tangent whereas they are separated by $A$. So now, assume that $z(A)$ and $z(B)$ are not nested. If $A$ and $B$ are transverse, there is nothing to prove; so assume that $A$ and $B$ are not transverse. Because $z(A')$ must be disjoint from $z(A)$, it is necessarily contained in $z(B)$. Then the same argument with $B,A'$ in place of $A,B$ leads to a contradiction. This concludes the proof of Claim~\ref{claim:PairwiseTangent}.

\medskip \noindent
It follows from the Helly property for convex subgraphs that the carriers of the hyperplanes appearing in $\mathcal{Z}$ globally intersect. Therefore, if $z'$ denote the map obtained from $z$ by replacing $z(J)$ with its complement for every hyperplane appearing in $\mathcal{Z}$, then $z'(A) \cap z'(B) \neq \emptyset$ for all the hyperplanes $A,B$ appearing in $\mathcal{Z}$. Since the hyperplanes appearing in $\mathcal{Z}$ all cross some ball of radius $1$, there are at most $N$ such hyperplanes. Thus, it suffices to show that $z'$ is an orientation in order to conclude the proof of our lemma.

\medskip \noindent
Let $A,B$ be two hyperplanes of $X$. If $A,B$ both appear in $\mathcal{Z}$, then we already know that $z'(A) \cap z'(B) \neq \emptyset$. If neither $A$ nor $B$ appear in $\mathcal{Z}$, then $z'(A) \cap z'(B) = z(A) \cap z(B) \neq \emptyset$. Otherwise, say that $A$ does not appear in $\mathcal{Z}$ and $\{B,B'\}\in \mathcal{Z}$ for some hyperplane $B'$. Because $z'(B)$ is the hyperplane containing $B'$, if $z'(A)=z(A)$ is disjoint from $z'(B)$, then $z(A)$ and $z(B')$ must be disjoint, which is impossible. Thus, $z'(A)$ and $z'(B)$ must intersect, proving that $z'$ is an orientation of $X$. 
\end{proof}

\begin{lemma}\label{lem:SwellLocallyFinite}
Let $X$ be a median graph and $Z$ another median graph obtained by swelling $X$. If $X$ is locally finite and if $Z$ has finite cubical dimension, then $Z$ is locally finite. 
\end{lemma}

\begin{proof}
We think of the vertices of $Z$ as orientations as described in Section~\ref{section:SwellingExplicit}. Fix a vertex $x \in X$ and a constant $R \geq 0$. We want to prove that there are only finitely many $z \in Z$ at distance $\leq R$ from $x$, or equivalently, which differ with $x$ only on $\leq R$ hyperplanes of $X$. Because $X$ is locally finite, it suffices to show that there exists a constant $K \geq 0$ depending only on $R$ and the cubical dimension of $Z$ such that all the hyperplanes on which $x$ and $z$ differ cross the ball $B(x,K)$ in $X$.

\medskip \noindent
So let $J$ be a hyperplane on which $x$ and $z$ differ. If $J$ does not cross a large ball centred at $x$, then there exist many hyperplanes separating $x$ from $J$. Moreover, we can exact from this collection a large subcollection of pairwise non-transverse hyperplanes (whose size only depends on the cubical dimension of $X$). However, there cannot be more than $R$ pairwise non-transverse hyperplanes $J_1, \ldots, J_n$ separating $x$ from $J$. Indeed, we know that $z(J)$ does not contain $x$; and, because $z$ and $x$ differ only on $R+1$ hyperplanes, there must exist some $1 \leq i \leq n-1$ such that $z(J_i)$ contains $x$. But then $z(J_i) \cap z(J)= \emptyset$, which implies that $\{J,J_i\}$ is a pair of swelled hyperplanes. In particular, they have to be tangent. But $J_n$ separates $J$ and $J_i$, hence a contradiction. 
\end{proof}

\begin{proof}[Proof of Proposition~\ref{prop:CocompactSwelling}.]
Because $X \to Y$ and $X \to Z$ are both chiasmatic, necessarily $Z \to Y$ is chiasmatic. This implies that the cubical dimension of $Z$ is bounded above by the cubical dimension of $Y$, hence finite. It follows from Lemma~\ref{lem:SwellCobounded} that $G$ acts coboundedly on $Z$ and from Lemma~\ref{lem:SwellLocallyFinite} that $Z$ is locally finite. Therefore, $G$ acts cocompactly on $Z$. 
\end{proof}

\noindent
We now turn our attention to foldings. Because we are only interested in chiasmatic maps here, it suffices to deal with foldings of tangent pairs of hyperplanes.

\begin{prop}\label{prop:CocompactFolding}
Let $G$ be a group acting on a median graph $X$. Assume that there are no hyperplane-inversions nor direct self-osculations. Fix two tangent hyperplanes $A,B$ and let $\eta : X \to Z$ be the folding of $X$ relative to $\mathcal{P}:= \{ g\{A,B\}, g\in G\}$. If $Z$ has finite cubical dimension and $G$ acts coboundedly on $X$, then $G$ acts coboundedly on $Z$. 
\end{prop}

\noindent
Given a median graph $X$, a \emph{hyperplane-inversion} is an isometry that stabilises a hyperplane but swaps the two halfspaces it delimits. A \emph{direct self-osculation} refers to an isometry $g \in \mathrm{Isom}(X)$ and two tangent hyperplanes $A,B$ such that $g$ sends $A$ to $B$ and $B$ in the halfspace delimited by $B$ opposite to $A$. For instance, the usual action of $\mathbb{Z}$ on the bi-infinite line admits direct self-osculation. It is worth noticing that hyperplane-inversions (resp. direct self-osculation) can always be avoided by replacing the median graph under consideration with its cubical subdivision. 

\medskip \noindent
Our proof of Proposition~\ref{prop:CocompactFolding} is inspired by the proof of \cite[Lemma~6.4]{beeker2018stallings}, and is based on our next two lemmas. 

\begin{lemma}\label{lem:FacingCollection}
Let $G$ be a group acting on a median graph $X$. Assume that there are no hyperplane-inversions nor direct self-osculations. Fix two tangent hyperplanes $A,B$ and let $\eta : X \to Z$ be the folding of $X$ relative to $\mathcal{P}:= \{ g\{A,B\}, g\in G\}$. For every hyperplane $J$ of $Z$, the collection $\eta^{-1}(J)$ is facing.
\end{lemma}

\noindent
In a median graph, a collection of hyperplanes $\mathcal{J}$ is \emph{facing} if no hyperplane of $\mathcal{J}$ separates two other hyperplanes from $\mathcal{J}$. 

\begin{proof}[Proof of Lemma~\ref{lem:FacingCollection}.]
First of all, notice that folding pairs of tangent hyperplanes yields a chiasmatic map. Therefore, given a hyperplane $J$ of $Z$, the hyperplanes in $\eta^{-1}(J)$ are pairwise non-transverse. This observation will be used several times in our argument.

\medskip \noindent
Assume that $\eta^{-1}(J)$ is not facing, i.e.\ there exist $P,Q,R \in \eta^{-1}(J)$ such that $R$ separates $P$ and $Q$. According to Theorem~\ref{thm:Foldings}, there exist
$$R_0:=P, \ R_1, \ldots, \ R_{n-1}, \ R_n:=Q$$
such that, for every $0 \leq i \leq n-1$, $\{R_i,R_{i+1}\}$ belongs to $\mathcal{P}$, say $\{R_i,R_{i+1}\}= g_i \{A,B\}$ for some $g_i \in G$. Because the $R_i$ are successively tangent, and because there is no transversality in $\eta^{-1}(J)$, there must exist some $1 \leq i \leq n-1$ such that $R_i$ separates $P$ and $Q$. 

\medskip \noindent
Let $j$ denote the largest index such that $j < i$ and such that $R_{j}$ lies in the same halfspace of $R_i$ as $P$. Necessarily, $R_{j+1}=R_i$, since otherwise $R_i$ would separate the tangent hyperplanes $R_{j}$ and $R_{j+1}$. Hence $j=i-1$. Similarly, we show that $i+1$ is the smallest index $k$ such that $i<k$ and such that $R_{k}$ lies in the same halfspace of $R_i$ as $Q$. Thus, $R_i$ separates $R_{i-1}$ and $R_{i+1}$. 

\medskip \noindent
Up to switching $A$ and $B$, two cases may happen:
\begin{itemize}
	\item $R_{i-1}=g_{i-1}A$, $R_i=g_{i-1}B= g_iB$, and $R_{i+1}=g_iA$. In this case, we find an element $g:=g_{i-1}^{-1}g_i$ stabilising $B$ such that $B$ separates $A$ and $gA$. In other words, $g$ inverts $B$.
	\item $R_{i-1}=g_{i-1}A$, $R_i=g_{i-1}B = g_iA$, and $R_{i+1}=g_iB$. In this case, we find element $g:=g_{i-1}^{-1}g_i$ that sends $A$ to $B$ such that $B$ separates $A$ and $gB$. Since $A$ and $B$ are tangent, we get a direct self-osculation.
\end{itemize}
Thus, we have proved that, if $\eta^{-1}(J)$ is not facing, then the action of $G$ on $X$ must admit a hyperplane-inversion or a direct self-osculation. 
\end{proof}

\begin{lemma}\label{lem:boundedaction}
Let $G$ be a group acting on a median graph $X$. Assume that there are no hyperplane-inversions nor direct self-osculations. Fix two tangent hyperplanes $A,B$ and let $\eta : X \to Z$ be the folding of $X$ relative to $\mathcal{P}:= \{ g\{A,B\}, g\in G\}$. If $Z$ has finite cubical dimension, then every vertex of $Z$ belongs to a cube intersecting $\eta(X)$.
\end{lemma}

\begin{proof}
In this argument, we think of $Z$ as the cubulation of the wallspace structure of $X$ as described in Section~\ref{section:FoldExplicit}.

\medskip \noindent
Fix a vertex $z \in Z$ and a maximal cube $Q$ containing it. For convenience, let $\mathcal{H}$ denote the set of pre-images under $\eta$ of the hyperplanes crossing $Q$. A good candidate for a principal orientation of $X$ sent in $Q$ by $\eta$ is the map $q$ defined as follows:
\begin{itemize}
	\item If $J \in \mathcal{H}$ and if $[J]$ has size one, then $q(J)$ is an arbitrary halfspace delimited by $J$ (e.g.\ the halfspace containing the fixed basepoint of $X$).
	\item If $J \in \mathcal{H}$ and if $[J]$ has size at least two, then we know from Lemma~\ref{lem:FacingCollection} that $[J]$ is a facing collection. So we can define $q(J)$ as the halfspace delimited by $J$ that contains the hyperplanes in $[J] \backslash \{J\}$.
	\item If $J \notin \mathcal{H}$, then it follows from the maximality of $Q$ that $\eta(J)$ is not transverse to at least one hyperplane of $Q$. Therefore, there exists $H \in \mathcal{H}$ such that $J$ is not transverse to any hyperplane in $[H]$. We define for $q(J)$ the halfspace delimited by $J$ containing $[H]$.
\end{itemize}
Our goal now is to verify that $q$ defines a principal orientation of $X$, and that, when $q$ is thought of as a vertex of $X$, $\eta(q) \in Q$. In the sequel, the following assertion, satisfied by $q$ by construction, will be often used, so we highlight it for the reader's convenience.

\begin{fact}\label{fact:ContainClass}
For every hyperplane $J \notin \mathcal{H}$, there exists some $H \in \mathcal{H}$ such that $[H] \subset q(J)$. 
\end{fact}

\noindent
First, let us verify that $q$ defines an orientation of $X$. If $q$ is not an orientation, then there exist two hyperplanes $A$ and $B$ of $X$ such that $q(A) \cap q(B)=\emptyset$. 

\medskip \noindent
If $A \notin \mathcal{H}$, then we know from Fact~\ref{fact:ContainClass} that $[H] \subset q(A)$ for some $H \in \mathcal{H}$. Because $A$ separates $B$ and $[H]$, $B$ cannot be transverse to any hyperplane in $[H]$, which implies that $\eta(B)$ and $\eta(H)$ are not transverse. So $B \notin \mathcal{H}$, and it follows from Fact~\ref{fact:ContainClass} that $[J] \subset q(B)$ for some $J \in \mathcal{H}$. But then $A$ and $B$ separate $[J]$ and $[H]$, contradicting the transversality of $\eta(J)$ and $\eta(H)$. We obtain the same contradiction if $B \notin \mathcal{H}$.

\medskip \noindent
Thus, we can assume that $A,B \in \mathcal{H}$. If $[A]=[B]$, it is clear by construction of $q$ that $q(A) \cap q(B)\neq \emptyset$, so $\eta(A)$ and $\eta(B)$ must be two distinct hyperplanes of $Q$. In particular, they are transverse, which means that there exist two transverse hyperplanes $A' \in [A]$ and $B' \in [B]$. Observe that, by construction of $q$, either $A'=A$ or $A' \subset q(A)$. Similarly, either $B'=B$ or $B' \subset q(B)$. We distinguish two cases. First, assume that $A'=A$. Because $A$ and $B$ are not transverse, necessarily $B' \neq B$, hence $B' \subset q(B)$. But then $B'$ cannot be transverse to $A$ as $A$ is disjoint from $q(B)$. Second, assume that $A' \neq A$, which amounts to $A' \subset q(A)$. Then $B'$ cannot be contained in $q(B)$, since otherwise it could not be transverse to $A'$. Therefore, $B'=B$, which is also impossible since $B$ is disjoint from $q(A)$. 

\medskip \noindent
This concludes the proof of the fact that $q$ defines an orientation of $X$. Now, we claim that $q$ is principal. Because $X$ has finite cubical dimension, this amounts to verifying that $\{ q(J), J \text{ hyperplane}\}$ satisfies the descending chain condition. So consider a sequence $q(J_1) \supsetneq q(J_2) \supsetneq \cdots$ where $J_1,J_2, \ldots$ are hyperplanes of $X$. Assume first that there are infinitely many $J_i$ in $\mathcal{H}$. Then there exists some $H \in \mathcal{H}$ such that $[J_i]=[H]$ for infinitely many $i$. But, by construction of $q$ and Lemma~\ref{lem:FacingCollection}, two hyperplanes in $[H]$ are not sent to properly nested halfspaces by $q$. So our sequence of halfspaces must be eventually constant. Next, assume that all (but finitely many of) the $J_i$ do not belong to $\mathcal{H}$. Then, as a consequence of Fact~\ref{fact:ContainClass}, there exists some $H \in \mathcal{H}$ for which $q(J_i) \supset [H]$ for every $i$. But, in a median graph, there cannot be an infinite decreasing sequence of halfspaces containing a given non-empty subset, so our sequence of halfspaces must be eventually constant. 

\medskip \noindent
Thus, we have proved that $q$ defines a principal orientation of $X$, or equivalently a vertex of $X$. It remains to verify that $\eta(q)$ belongs to $Q$. Let $J$ be a hyperplane of $Z$ that does not cross $Q$. Fix a pre-image $K \in \eta^{-1}(J)$. Because $K \notin \mathcal{H}$, it follows from Fact~\ref{fact:ContainClass} that $q(K) \supset [H]$ for some $H \in \mathcal{H}$. Therefore, $\eta(q)$ lies in the same halfspace delimited by $J$ as $H$, or equivalently as $Q$. Thus, $\eta(q)$ belongs to the intersection of all the halfspaces containing $Q$, which amounts to saying that $\eta(q)$ belongs to $Q$, as desired. 
\end{proof}

\begin{proof}[Proof of Proposition~\ref{prop:CocompactFolding}.]
We know by assumption that $Z$ has finite cubical dimension $D$, so Lemma~\ref{lem:boundedaction} applies and shows that every vertex of $Z$ lies at distance $\leq D$ from $\eta(X)$. But $G$ acts coboundedly on $\eta(X)$, since it acts coboundedly on $X$. Therefore, $G$ must act coboundedly on $Z$ as well. 
\end{proof}

\begin{proof}[Proof of Theorem~\ref{thm:FreeSub}.]
Let $F \leq G$ be a free subgroup of finite rank. Assume that every hyperplane skewered by $F$ has a finitely generated $F$-stabiliser. We first construct an $F$-invariant convex subgraph $Y \subset X$ all of whose hyperplanes are skewered by $F$. This can be done as follows. 

\begin{fact}\label{fact:Skewer}
Let $F$ be a finitely generated group acting on a median graph $X$. There exists an $F$-invariant convex subgraph $Y \subset X$ all of whose hyperplanes are skewered by~$F$. 
\end{fact}

\noindent
We begin by considering the convex hull of an orbit of $F$. It is standard that, because $F$ is finitely generated, there are only finitely many $F$-orbits of hyperplanes in this hull. Then \cite[Proposition~3.5]{MR2827012} applies and provides an \emph{essential core}, which satisfies the desired property according to \cite[Proposition~3.2]{MR2827012}. This proves Fact~\ref{fact:Skewer}.

\medskip \noindent
Next, we fix a basepoint $o \in Y$ and a free basis $S \subset F$. Let $T$ denote the Cayley tree $\mathrm{Cayl}(F,S)$. We construct an $F$-equivariant chiasmatic map $T' \to Y$ from some subdivision of $T'$ of $T$ as follows. For every element $g \in F$ and generator $s \in S$, we subdivide the edge $[g,gs]$ of $T$ in order to get a path of length $d(o,so)$. Let $T'$ denote the tree thus obtained. Then the map $T' \to Y$ is defined by sending every vertex $g$ of $T \subset T'$ to $go \in Y$ and by sending each path $[g,gs]$, where $g \in F$ and $s \in S$, to $g [o,so]$ where $[o,so]$ is a geodesic of $Y$ we fixed once for all. 

\medskip \noindent
Apply Theorem~\ref{thm:BF} to find finitely many $F$-equivariant foldings and swellings $T \to \cdots \to Z$ in order to factor $T \to Y$ as an $F$-equivariant isometric embedding $Z \hookrightarrow Y$. Because $T \to Y$ is chiasmatic, our foldings, swellings, and isometric embedding are all chiasmatic as well. In particular, this implies that the cubical dimension of $Z$ is bounded above by the cubical dimension of $X$, and therefore must be finite. Propositions~\ref{prop:CocompactSwelling} and~\ref{prop:CocompactFolding} then apply and show that the image of $Z$ in $X$ yields a median subgraph on which $F$ acts cocompactly. Thus, $F$ is median-cocompact. The converse of the theorem is clear. 
\end{proof}

\subsection{Resolutions}

\noindent
Another way to construct chiasmatic maps comes from \emph{resolutions}, a tool initially used for actions on trees and generalised to median graphs in \cite{MR3546458}. The construction goes as follows.

\medskip \noindent
Let $G$ be a finitely presented group acting on a median graph $X$. 
\begin{itemize}
	\item Fix a basepoint $o \in X$. Also, fix a finite presentation $\langle S \mid R \rangle$ for $G$ and let $K$ denote the corresponding $2$-complex. The universal cover $\tilde{K}$ of $K$ is a simply connected $2$-complex on which $G$ acts geometrically, and its one-skeleton coincides with the Cayley graph $\mathrm{Cayl}(G,S)$. 
	\item For all $g \in G$ and $s \in S$, we subdivide the edge $[g,gs]$ of $\tilde{K}$ into a segment of length $d(o,so)$, and we send it to $g[o,so]$ for some fixed geodesic $[o,so]$ in $X$. For every relation $r =s_1 \cdots s_n$ in $R$ and for every $g \in G$, we tile the $2$-cell of $\tilde{K}$ bounded by $g,gs_1, \ldots, gs_1 \cdots s_n = g$ appropriately and send it to a minimal disc diagram bounded by the loop $go, gs_1o, \ldots, gs_1 \cdots s_no$ in the square-completion $X_\square$ of $X$ (i.e.\ the square complex obtained from $X$ by filling with squares all the $4$-cycles). Thus, one obtains a square complex $\tilde{K}'$ of $\tilde{K}$ and a $G$-equivariant map $\tilde{K}' \to X_\square$ that sends cells to cells (possibly of smaller dimensions). 
	\item To every hyperplane $J$ of $X$ corresponds a \emph{dual curve} in $X_\square$ that connects the midpoints of the edges in $J$ along straight segments in squares. Connected components of pre-images of dual curves under $\tilde{K}' \to X_\square$ are \emph{tracks}. They are embedded graphs defining a wallspace structure on $\tilde{K}'$. Let $C$ denote its cubulation. According to \cite[Proposition~3.2]{MR3546458}, the map $\tilde{K}' \to X_\square$ induces a $G$-equivariant chiasmatic map $C \to X$, called a \emph{resolution} of $G \curvearrowright X$. 
\end{itemize}
Let us record the following observation, which is essentially contained in the proof of \cite[Theorem~7.2]{beeker2018stallings}. 

\begin{thm}\label{thm:Resolution}
Let $G$ be a group acting geometrically on a median graph $X$. Given a finitely presented subgroup $H \leq G$, if hyperplane-stabilisers of $H$ are finitely generated and if the action $H \curvearrowright X$ admits a cobounded resolution, then $H$ is median-cocompact.
\end{thm}

\begin{proof}
Let $C \to X$ be a cobounded resolution. As a consequence of Theorem~\ref{thm:BF}, $C \to X$ factors as $C \to Z \to X$ where $C \to Z$ is a composition of finitely many $G$-atomic foldings and swellings. Because $C \to X$ is chiasmatic, so are $C \to Z$ and $Z \to X$. In particular, this implies that the cubical dimension of $Z$ is bounded above by the cubical dimension of $X$, and therefore must be finite. Propositions~\ref{prop:CocompactSwelling} and~\ref{prop:CocompactFolding} then applies and shows that the image of $Z$ in $X$ yields a median subgraph on which $H$ acts cocompactly. Thus, $H$ is median-cocompact, as desired. 
\end{proof}

\noindent
It is worth noticing that Theorem~\ref{thm:FreeSub} can also be obtained from Theorem~\ref{thm:Resolution}. Indeed, the resolutions obtained for free groups from presentations with no relations provide actions on subdivided Cayley trees, and such actions are clearly cobounded. More generally, as an intersection consequence of Theorem~\ref{thm:Resolution}, let us mention the following generalisation of Theorem~\ref{thm:FreeSub} (which is essentially contained in the proof of \cite[Theorem~1.2]{beeker2018stallings}): 

\begin{cor}\label{cor:LocallyQC}
Let $G$ be a group acting properly and cocompactly on a median graph $X$. A locally quasiconvex hyperbolic subgroup $H\leq G$ is median-cocompact if and only if every hyperplane skewered by $H$ has a finitely generated $H$-stabiliser. 
\end{cor}

\noindent
Recall that a hyperbolic is \emph{locally quasiconvex} if all its finitely generated subgroups are quasiconvex. This includes for instance free groups and hyperbolic surface groups, but also many two-dimensional hyperbolic groups. 

\begin{proof}[Sketch of proof of Corollary~\ref{cor:LocallyQC}.]
According to Fact~\ref{fact:Skewer}, there exists an $H$-invariant convex subgraph $Y \subset X$ all of whose hyperplanes are skewered by $H$. We claim that every resolution $Z \to Y$ of $H \curvearrowright Y$ is cobounded. Indeed, we know that hyperplane-stabilisers in $Z$ are finitely generated (this is a fundamental property satisfied by resolutions, see \cite{MR3546458}). Thus, they must be quasiconvex by assumption. But, as shown in \cite{MR1438181}, a cubulation with respect to finitely many quasiconvex subgroups is automatically cocompact. The converse of the corollary is clear.
\end{proof}

\noindent
As another consequence of Theorem~\ref{thm:Resolution}, let us mention the following statement (which is essentially contained in the proof of \cite[Theorem~1.2]{beeker2018stallings}):

\begin{cor}
Let $G$ be a group acting geometrically on a median graph $X$ of cubical dimension two. A finitely presented subgroup $H \leq G$ is median-cocompact  if and only if every hyperplane skewered by $H$ has a finitely generated $H$-stabiliser. 
\end{cor}

\begin{proof}
The corollary is an immediate consequence of Fact~\ref{fact:Skewer} and of Theorem~\ref{thm:Resolution} combined with \cite[Theorem~1.1]{MR3546458}, which shows that, because $X$ has cubical dimension two, every resolution is cobounded. 
\end{proof}

\section{Comparison with other fold-like operations}\label{section:comparison}

\noindent
A few folding techniques for median graphs are already available in the literature, see \cite{BrownPhD, beeker2018stallings, dani2021subgroups, ben2022folding}. In this final section, we compare our framework with the formalisms of \cite{beeker2018stallings} and \cite{ben2022folding} (leaving aside \cite{BrownPhD} and \cite{dani2021subgroups} as they deal with specific families of median graphs, namely graphs of graphs and Cayley graphs of right-angled Coxeter groups).

\paragraph{Comparison with \cite{beeker2018stallings}.} Recall that a \emph{pocset} $(\mathcal{P},\leq, ^\ast)$ is the data of a partially ordered set $(\mathcal{P},\leq)$ equipped with an involution $^\ast : \mathcal{P} \to \mathcal{P}$ such that, for every $\mathfrak{h} \in \mathcal{P}$, $\mathfrak{h}$ and $\mathfrak{h}^\ast$ are distinct and not $\leq$-comparable. The typical example to keep in mind is the set of halfspaces of a median graph, ordered by inclusion, equiped with complementation. Conversely, pocsets naturally define (disjoint unions of) median graphs \cite{MR1347406, Roller}. Following this equivalence, most of the terminology of median graphs naturally extends to pocsets. For instance, a \emph{hyperplane} in our pocset refers to a pair $\hat{\mathfrak{h}}:= \{ \mathfrak{h}, \mathfrak{h}^\ast\}$, where $\mathfrak{h} \in \mathcal{H}$. 

\medskip \noindent
In \cite{beeker2018stallings}, the authors define various maps between pocsets and a folding operation, and they prove a factorisation theorem. Namely:

\begin{thm}[\cite{beeker2018stallings}]
Let $G$ be a group acting on two pocsets $\mathcal{H},\mathcal{H}'$ and $f : \mathcal{H} \to \mathcal{H}'$ a $G$-equivariant resolution. If $G$ has only finitely many orbits in $\mathcal{H}$ and if hyperplane-stabilisers in $\mathcal{H}'$ are finitely generated, then $\mathcal{H} \to \mathcal{H}'$ factors as $\mathcal{H} \to \mathcal{K} \to \mathcal{H}'$ where $\mathcal{H} \to \mathcal{K}$ is a finite $G$-equivariant folding sequence and where $\mathcal{K} \to \mathcal{H}'$ is a $G$-equivariant embedding.
\end{thm}

\noindent
The statement requires some definitions. A map $f : \mathcal{H} \to \mathcal{H}'$ between two pocsets is \emph{admissible} if it satisfies the following conditions:
\begin{itemize}
	\item $f(\mathfrak{h}^\ast)= f(\mathfrak{h})^\ast$ for every $\mathfrak{h} \in \mathcal{H}$;
	\item if $\mathfrak{h},\mathfrak{k} \in \mathcal{H}$ are transverse, then so are $f(\mathfrak{h})$ and $f(\mathfrak{k})$;
	\item for all facing halfspaces $\mathfrak{h},\mathfrak{k} \in \mathcal{H}$, if $f(\hat{\mathfrak{h}})= f(\hat{\mathfrak{k}})$ and if $\hat{\mathfrak{h}},\hat{\mathfrak{k}}$ are not separated by a hyperplane in $f^{-1}(f(\hat{\mathfrak{h}}))$, then $f(\mathfrak{h})= f(\mathfrak{k})$;
	\item every hyperplane $\hat{\mathfrak{h}}$ in $\mathcal{H}'$ that is not the image of hyperplane of $\mathcal{H}$ under $f$ admits an orientation $\mathfrak{h}$ compatible with all the halfspace in $f(\mathcal{H})$.
\end{itemize}
An admissible map $f : \mathcal{H} \to \mathcal{H}'$ is an \emph{embedding of pocsets} if $f$ is injective and if, for all $\mathfrak{h},\mathfrak{k} \in \mathcal{H}$, $f(\mathfrak{h}) \leq f(\mathfrak{k})$ implies $\mathfrak{h} \leq \mathfrak{k}$. An admissible map $f : \mathcal{H} \to \mathcal{H}'$ is a \emph{resolution of pocsets} if $(\mathcal{H} / \sim_f) \to \mathcal{H}'$ is an embedding of pocsets, where $\sim_f$ is the \emph{admissible equivalence relation} (see \cite[Definition~3.2]{beeker2018stallings}) defined by: $\mathfrak{h} \sim_f \mathfrak{k}$ whenever $f(\mathfrak{h})= f(\mathfrak{k})$. 

\medskip \noindent
Given a group $G$ acting on two pocsets $\mathcal{H},\mathcal{H}'$ and given a $G$-equivariant resolution of pocsets $f : \mathcal{H} \to \mathcal{H}'$, two facing halfspaces $\mathfrak{h}_1$ and $\mathfrak{h}_2$ of $\mathcal{H}$ are \emph{elementary foldable} if $\mathfrak{h}_1 \sim_f \mathfrak{h}_2$ and if there are no facing pairs $\mathfrak{k}_1 \sim_f \mathfrak{k}_2$ satisfying $\mathfrak{k}_1 \leq \mathfrak{h}_1$ and $\mathfrak{k}_2 \leq \mathfrak{h}_2$. An \emph{elementary fold} is a quotient of the form $\mathcal{H}/ \sim$ where $\sim$ is the minimal $G$-invariant and $^\ast$-invariant equivalence relation generated by identifying an elementary foldable pair. 

\medskip \noindent
In order to compare this formalism with our approach, let us describe the definitions above in terms of median graphs. 

\medskip \noindent
First, it follows from \cite[Lemma~4.4]{beeker2018stallings} that a resolution between two pocsets can be realised as a map between the corresponding median graphs. Such a map is clearly chiasmatic. Conversely, it can be verified that a chiasmatic map between two median graphs induces a resolution between the corresponding pocsets. Embeddings of pocsets then correspond to isometric embeddings. Consequently, one can say that \cite{beeker2018stallings} deals with chiasmatic maps between median graphs. We meet a first difference with our framework: since our results apply to parallel-preserving maps, our scope is (slightly) more general.

\medskip \noindent
Using the point of view of median graphs, the elementary folds described above amount to the following. Given a group $G$ acting on two median graphs $X,Y$ and given a $G$-equivariant chiasmatic map $\psi : X \to Y$, if $\psi$ is not an isometric embedding, then there exist two hyperplanes $J_1,J_2$ having the same image under $\psi$. We can choose $J_1$ and $J_2$ at minimal distance. Fix orientations $J_1^+,J_2^+$ of $J_1,J_2$ such that the halfspaces $J_1^+,J_2^+$ contain respectively $J_2,J_1$. Then $X \to Y$ factors as $X \to Z \to Y$ where $Z$ is the cubulation of the pocset obtained as the quotient of the pocset of $X$ by the smallest reasonable $G$-invariant equivalence relation identifying $J_1^+$ and $J_2^+$. The map $X \to Z$ is an elementary fold, and, according to \cite{beeker2018stallings}, every chiasmatic map between median graphs factors as a sequence of elementary folds followed by an isometric embedding. Compare with Theorem~\ref{thm:BigFactor}. 

\medskip \noindent
We emphasize that, in such an elementary fold, the pair $\{J_1,J_2\}$ may not be tangent and cannot be transverse. Thus, even though our two notions of fold agree for tangent pairs, neither one subsumes the other: our foldings may identify two transverse hyperplanes, the foldings from \cite{beeker2018stallings} may identify two hyperplanes that are not in contact. Nevertheless, due to the universal properties satisfied by our foldings and swellings, an elementary fold of \cite{beeker2018stallings} factors through finitely many of our foldings and swellings. Roughly speaking, in order to identify two hyperplanes $J_1$ and $J_2$ far apart, we can make $J_1$ transverse to all the hyperplanes separating $J_1$ and $J_2$ thanks to finitely many swellings; this process makes $J_1$ and $J_2$ tangent, so a folding can identify them. Consequently, the factorisation provided by Theorem~\ref{thm:BigFactor} is, in some sense, sharper.

\medskip \noindent
Of course, another difference is that we introduced two operations: foldings and swellings. Adding swellings allows us to factor parallel-preserving maps through isometric embeddings with convex images, something which is not reachable from the techniques of \cite{beeker2018stallings}. Isometric embeddings with convex images between median graphs correspond to isometric embeddings between CAT(0) cube complexes, a point of view taken in \cite{ben2022folding} and which we discuss now.

\paragraph{Comparison with \cite{ben2022folding}.} A cube complex is \emph{nonpositively curved} whenever its vertices have simplicial flag links. When endowed with their $\ell^2$-metrics, nonpositively curved cube complexes are locally CAT(0). Following \cite{dani2021subgroups}, \cite{ben2022folding} defines a few elementary operations on cube complexes and shows that cubical maps between nonpositively curved cube complexes factor as locally isometric embeddings through a (possibly infinite) sequence of these elementary operations. More precisely, consider the following transformations.
\begin{description}
	\item[(Folding)] Identify two edges that share an endpoint.
	\item[(Cube identification)] Identify two cubes with the same one-skeleton.
	\item[(Cube attachment)] If $e_1,\ldots, e_n$ are $n$ edges with a common origin $o$ that pairwise span squares, glue an $n$-cube along $o$ and $e_1, \ldots, e_n$. 
\end{description}
The main result of \cite{ben2022folding} shows that these operations are sufficient to turn cubical maps into locally isometric embeddings.

\begin{thm}[\cite{ben2022folding}]
Let $X$ be a connected cube complex and $Y$ a compact nonpositively curved cube complex. Every cubical map $X \to Y$ factors as $X \to Z \to Y$ where $Z \to Y$ is a locally isometric embedding and where $X \to Z$ is a (possibly infinite) sequence of foldings, cube identifications, and cube attachments. 
\end{thm}

\noindent
Thus, \cite{ben2022folding} works in quotients and use CAT(0) metrics. Taking universal covers and the median point of view, the result can be reformulated as follows. Let $G$ be a torsion-free group acting properly on two median graphs $X,Y$ and let $X \to Y$ be a $G$-equivariant combinatorial map. If $G$ acts cocompactly on $Y$, then $X \to Y$ factors as $X \to Z \to Y$ where $Z \to Y$ is an isometric embedding with convex image and where $X \to Z$ is a (possibly infinite) sequence of $G$-invariant elementary transformations. 

\medskip \noindent
We obtained a similar result in our framework, see Theorem~\ref{thm:BigFactorTwo}. But there are major differences. On the one hand, Theorem~\ref{thm:BigFactorTwo} requires the maps between median graphs to be parallel-preserving, which is (much) more restrictive, but this what allows us to keep some control on the median geometry, which is not the case in \cite{ben2022folding}, as mentioned below. On the other hand, due to the fact that the authors work with fundamental groups of compact nonpositively curved cube complexes, \cite{ben2022folding} can only deal with free actions, and consequently with torsion-free groups. For instance, the results of \cite{dani2021subgroups}, proved for right-angled Coxeter groups, do not follow from \cite{ben2022folding}. Theorem~\ref{thm:BigFactorTwo} has no restriction on the actions under consideration. 

\medskip\noindent
Also, the elementary transformations described above do not preserve in general the geometric structure of cube complexes: typically, applying a folding or a cube attachment to a nonpositively curved cube complex produces a cube complex that is not nonpositively curved.This has to be compared with our Theorem~\ref{thm:BigFactorTwo}, which factors parallel-preserving maps between median graphs through maps between median graphs. In other words, our foldings and swellings can be decomposed further through the elementary operations of \cite{ben2022folding}, but the price to pay is that the median geometry is lost along the way. 

\medskip \noindent
Focusing on CAT(0) geometry, \cite{ben2022folding} cannot reach some results only provable using median geometry. For instance, \cite{ben2022folding} deals with convex-cocompact subgroups (i.e.\ subgroups acting cocompactly on convex subgraphs) but does not cover median-cocompact subgroups, a family of subgroups that is more general and more suited (for instance, cyclic subgroups may not be convex-cocompact whereas virtually abelian subgroups are always median-cocompact). Thanks to our foldings and swellings, both convex- and median-cocompact subgroups belong to the scope of our techniques.

\addcontentsline{toc}{section}{References}

\bibliographystyle{alpha}
{\footnotesize\bibliography{MedianFoldings}

\Address

\end{document}